\documentclass[reqno]{amsart}
\usepackage{amsfonts}

\newtheorem{theorem}{Theorem}
\theoremstyle{plain}

\newtheorem{case}{Case}

\newtheorem{lemma}{Lemma}

\newtheorem{proposition}{Proposition}
\newtheorem{remark}{Remark}

\numberwithin{equation}{section}
\numberwithin{theorem}{section}
\numberwithin{algorithm}{section}
\numberwithin{axiom}{section}
\numberwithin{case}{section}
\numberwithin{claim}{section}
\numberwithin{conclusion}{section}
\numberwithin{condition}{section}
\numberwithin{conjecture}{section}
\numberwithin{corollary}{section}
\numberwithin{criterion}{section}
\numberwithin{definition}{section}
\numberwithin{example}{section}
\numberwithin{exercise}{section}
\numberwithin{lemma}{section}
\numberwithin{notation}{section}
\numberwithin{problem}{section}
\numberwithin{proposition}{section}
\numberwithin{remark}{section}
\numberwithin{solution}{section}

\input{tcilatex}

\begin{document}
\title{$Q$ curvature on a class of manifolds with dimension at least $5$}
\author{Fengbo Hang}
\address{Courant Institute, New York University, 251 Mercer Street, New York
NY 10012}
\email{fengbo@cims.nyu.edu}
\author{Paul C. Yang}
\address{Department of Mathematics, Princeton University, Fine Hall,
Washington Road, Princeton NJ 08544}
\email{yang@math.princeton.edu}
\thanks{The research of Yang is supported by NSF grant 1104536.}
\date{}

\begin{abstract}
For a smooth compact Riemannian manifold with positive Yamabe invariant,
positive $Q$ curvature and dimension at least $5$, we prove the existence of
a conformal metric with constant $Q$ curvature. Our approach is based on the
study of extremal problem for a new functional involving the Paneitz
operator.
\end{abstract}

\maketitle

\section{Introduction\label{sec1}}

Recall the definition of the $4$th order Paneitz operator and its associated 
$Q$ curvature \cite{B,P}: when $\left( M,g\right) $ is a smooth compact $n$
dimensional Riemannian manifold with $n\geq 3$, the $Q$ curvature is given by%
\begin{eqnarray}
Q &=&-\frac{1}{2\left( n-1\right) }\Delta R-\frac{2}{\left( n-2\right) ^{2}}%
\left\vert Rc\right\vert ^{2}+\frac{n^{3}-4n^{2}+16n-16}{8\left( n-1\right)
^{2}\left( n-2\right) ^{2}}R^{2}  \label{eq1.1} \\
&=&-\Delta J-2\left\vert A\right\vert ^{2}+\frac{n}{2}J^{2}.  \notag
\end{eqnarray}%
Here $R$ is the scalar curvature, $Rc$ is the Ricci tensor and%
\begin{equation}
J=\frac{R}{2\left( n-1\right) },\quad A=\frac{1}{n-2}\left( Rc-Jg\right) .
\label{eq1.2}
\end{equation}%
The Paneitz operator is given by%
\begin{eqnarray}
&&P\varphi  \label{eq1.3} \\
&=&\Delta ^{2}\varphi +\frac{4}{n-2}\func{div}\left( Rc\left( \nabla \varphi
,e_{i}\right) e_{i}\right) -\frac{n^{2}-4n+8}{2\left( n-1\right) \left(
n-2\right) }\func{div}\left( R\nabla \varphi \right) +\frac{n-4}{2}Q\varphi 
\notag \\
&=&\Delta ^{2}\varphi +\func{div}\left( 4A\left( \nabla \varphi
,e_{i}\right) e_{i}-\left( n-2\right) J\nabla \varphi \right) +\frac{n-4}{2}%
Q\varphi .  \notag
\end{eqnarray}%
Here $e_{1},\cdots ,e_{n}$ is a local orthonormal frame with respect to $g$.
When $n\neq 4$, under a conformal change of the metric, the operator
satisfies%
\begin{equation}
P_{\rho ^{\frac{4}{n-4}}g}\varphi =\rho ^{-\frac{n+4}{n-4}}P_{g}\left( \rho
\varphi \right) .  \label{eq1.4}
\end{equation}%
This is similar to the conformal Laplacian operator, which appears naturally
when considering transformation law of the scalar curvature under conformal
change of metric in dimension greater than $2$ (\cite{LP}). As a consequence
we have%
\begin{equation}
P_{\rho ^{\frac{4}{n-4}}g}\varphi \cdot \psi d\mu _{\rho ^{\frac{4}{n-4}%
}g}=P_{g}\left( \rho \varphi \right) \cdot \rho \psi d\mu _{g}.
\label{eq1.5}
\end{equation}%
Here $\mu _{g}$ is the measure associated with metric $g$.

In dimension $4$, the Paneitz operator is given by%
\begin{equation}
P\varphi =\Delta ^{2}\varphi +2\func{div}\left( Rc\left( \nabla \varphi
,e_{i}\right) e_{i}\right) -\frac{2}{3}\func{div}\left( R\nabla \varphi
\right) ,  \label{eq1.6}
\end{equation}%
and its conformal covariance property takes the form%
\begin{equation}
P_{e^{2w}g}\varphi =e^{-4w}P_{g}\varphi .  \label{eq1.7}
\end{equation}%
Following the basic work \cite{CGY} in dimension $4$ on the $4$th order $Q$
curvature equation, there has been several studies on this equation in
dimension $3$ by \cite{HY1, XY2,YZ}, and in dimensions greater than $4$ by 
\cite{DHL, HeR1, HeR2, HuR, QR1, QR2}.

While it is important to determine conditions under which the Paneitz
operator is positive, we discover that it is sufficient for our purpose in
this article to determine when its Green's function is positive. This is a
property that is conformally invariant: observe that by (\ref{eq1.4}),%
\begin{equation}
\ker P_{g}=0\Leftrightarrow \ker P_{\rho ^{\frac{4}{n-4}}g}=0,  \label{eq1.8}
\end{equation}%
and under this assumption, the Green's functions $G_{P}$ satisfy the
transformation law%
\begin{equation}
G_{P,\rho ^{\frac{4}{n-4}}g}\left( p,q\right) =\rho \left( p\right)
^{-1}\rho \left( q\right) ^{-1}G_{P,g}\left( p,q\right) .  \label{eq1.9}
\end{equation}

In analogy with the preliminary study of the classical Yamabe problem (\cite%
{LP}), the first question would be whether one can find a conformal
invariant condition for the existence of a conformal metric with positive $Q$
curvature. In the case Yamabe invariant $Y\left( g\right) $ is positive, the
existence of a conformal metric with positive $Q$ curvature is equivalent to
the requirements that $\ker P=0$ and the Green's function $G_{P}>0$ (\cite%
{HY4}).

The basic question of interest is to find constant $Q$ curvature metric in a
conformal class, in the same spirit as Yamabe problem. The main aim of the
present article is to prove the following

\begin{theorem}
\label{thm1.1}Let $\left( M,g\right) $ be a smooth compact $n$ dimensional
Riemannian manifold with $n\geq 5$, $Y\left( g\right) >0$, $Q\geq 0$ and not
identically zero, then $\ker P=0$, the Green's function of $P$ is positive
and there exists a conformal metric $\widetilde{g}$ with $\widetilde{Q}=1$.
\end{theorem}

\begin{remark}
\label{rmk1.1}Let $\left( M^{n},g\right) $ be a smooth compact Riemannian
manifold with $n\geq 5$, $Y\left( g\right) >0$. Denote $L=-\frac{4\left(
n-1\right) }{n-2}\Delta +R$ as the conformal Laplacian operator and for $%
p\in M$, $G_{L,p}$ as the Green's function of $L$ with pole at $p$. Define 
\begin{equation*}
\Gamma _{1}\left( p,q\right) =2^{\frac{n-6}{n-2}}n^{-\frac{2}{n-2}}\left(
n-1\right) ^{\frac{n-4}{n-2}}\left( n-2\right) ^{-3}\omega _{n}^{-\frac{2}{%
n-2}}G_{L}\left( p,q\right) ^{\frac{n-4}{n-2}}\left\vert Rc_{G_{L,p}^{\frac{4%
}{n-2}}g}\right\vert _{g}^{2}\left( q\right) .
\end{equation*}%
Here $\omega _{n}$ is the volume of unit ball in $\mathbb{R}^{n}$, $%
G_{L}\left( p,q\right) =G_{L,p}\left( q\right) $. The associated integral
operator $T_{\Gamma _{1}}$ is given by%
\begin{equation*}
T_{\Gamma _{1}}\left( \varphi \right) \left( p\right) =\int_{M}\Gamma
_{1}\left( p,q\right) \varphi \left( q\right) d\mu \left( q\right)
\end{equation*}%
for any nice function $\varphi $ on $M$. In \cite{HY5}, it is shown that the
spectrum $\sigma \left( T_{\Gamma _{1}}\right) $ and spectral radius $%
r_{\sigma }\left( T_{\Gamma _{1}}\right) $ are conformal invariants,
moreover the following statements are equivalent:

\begin{enumerate}
\item there exists a conformal metric $\widetilde{g}$ with $\widetilde{Q}>0$.

\item $\ker P=0$ and the Green's function of Paneitz operator $G_{P}\left(
p,q\right) >0$ for $p\neq q$.

\item $\ker P=0$ and there exists $p\in M$ such that $G_{P}\left( p,q\right)
>0$ for $q\neq p$.

\item $r_{\sigma }\left( T_{\Gamma _{1}}\right) <1$.
\end{enumerate}

Under the assumption $Q\geq 0$ and not identically zero, we have $r_{\sigma
}\left( T_{\Gamma _{1}}\right) <1$.
\end{remark}

The fundamental difficulty of the lack of maximum principle in this $4$th
order equation has recently been overcome by the work in \cite{GM}.
Following this development, similar results in dimension $3$ were proved in 
\cite{HY3,HY4} (see also closely related \cite{HY2}). Dimension $4$ case
does not suffer from this difficulty and was treated in many articles like 
\cite{CY,DM,FR} and so on. For a locally conformally flat manifold with
positive Yamabe invariant and Poincare exponent less than $\frac{n-4}{2}$
(see \cite{SY}), Theorem \ref{thm1.1} was proved in \cite{QR2} by apriori
estimates and connecting the equation to Yamabe equation through a path of
integral equations. Under the slightly more stringent conditions $R>0$ and $%
Q>0$, Theorem \ref{thm1.1} was proved in \cite{GM} through the study of a
non-local flow. Here we will derive Theorem \ref{thm1.1} by maximizing a
functional (see (\ref{eq1.16}) and (\ref{eq2.2})) involving the Paneitz
operator (see Theorem \ref{thm1.3} for more details).

For $u,v\in C^{\infty }\left( M\right) $, we denote the quadratic form
associated to $P$ as%
\begin{eqnarray}
&&E\left( u,v\right)  \label{eq1.10} \\
&=&\int_{M}Pu\cdot vd\mu  \notag \\
&=&\int_{M}\left( \Delta u\Delta v-\frac{4}{n-2}Rc\left( \nabla u,\nabla
v\right) +\frac{n^{2}-4n+8}{2\left( n-1\right) \left( n-2\right) }R\nabla
u\cdot \nabla v\right.  \notag \\
&&\left. +\frac{n-4}{2}Quv\right) d\mu  \notag \\
&=&\int_{M}\left( \Delta u\Delta v-4A\left( \nabla u,\nabla v\right) +\left(
n-2\right) J\nabla u\cdot \nabla v+\frac{n-4}{2}Quv\right) d\mu ,  \notag
\end{eqnarray}%
and%
\begin{equation}
E\left( u\right) =E\left( u,u\right) .  \label{eq1.11}
\end{equation}%
By the integration by parts formula in (\ref{eq1.10}) we know that $E\left(
u,v\right) $ extends continuously to $u,v\in H^{2}\left( M\right) $.

To find the metric $\widetilde{g}$ in Theorem \ref{thm1.1}, we write $%
\widetilde{g}=\rho ^{\frac{4}{n-4}}g$, then the equation $\widetilde{Q}=1$
becomes%
\begin{equation}
P_{g}\rho =\frac{n-4}{2}\rho ^{\frac{n+4}{n-4}},\quad \rho \in C^{\infty
}\left( M\right) ,\rho >0.  \label{eq1.12}
\end{equation}%
Let%
\begin{equation}
Y_{4}\left( g\right) =\inf_{u\in H^{2}\left( M\right) \backslash \left\{
0\right\} }\frac{E\left( u\right) }{\left\Vert u\right\Vert _{L^{\frac{2n}{%
n-4}}}^{2}},  \label{eq1.13}
\end{equation}%
then $Y_{4}\left( \tau ^{\frac{4}{n-4}}g\right) =Y_{4}\left( g\right) $ for
any positive smooth function $\tau $. Hence $Y_{4}\left( g\right) $ is a
conformal invariant. If $\left( M,g\right) $ is not locally conformally flat
and $n\geq 8$, or $\left( M,g\right) $ is locally conformally flat with $%
Y\left( g\right) >0$, $\ker P=0$ and the Green's function of $P$, $G_{P}>0$,
or $n=5,6,7$ with $Y\left( g\right) >0$, $\ker P=0$ and $G_{P}>0$, one can
show $Y_{4}\left( g\right) $ is achieved (see \cite{ER,GM,R}), but in
general it is difficult to know whether the minimizer is positive. Under the
additional assumption $Y_{4}\left( g\right) >0$ and $G_{P}>0$, it was
observed in \cite{R} that the minimizer cannot change sign. Combining this
with the positivity criterion of Green's function in \cite{HY4}, we arrive at

\begin{theorem}
\label{thm1.2}Let $\left( M,g\right) $ be a smooth compact $n$ dimensional
Riemannian manifold with $n\geq 5$, $Y\left( g\right) >0,Y_{4}\left(
g\right) >0,Q\geq 0$ and not identically zero, then

\begin{enumerate}
\item $Y_{4}\left( g\right) \leq Y_{4}\left( S^{n}\right) $, and equality
holds if and only if $\left( M,g\right) $ is conformally diffeomorphic to
the standard sphere.

\item $Y_{4}\left( g\right) $ is always achieved. Any minimizer must be
smooth and cannot change sign. In particular we can find a constant $Q$
curvature metric in the conformal class.

\item If $\left( M,g\right) $ is not conformally diffeomorphic to the
standard sphere, then the set of all minimizers $u$ for $Y_{4}\left(
g\right) $, after normalizing with $\left\Vert u\right\Vert _{L^{\frac{2n}{%
n-4}}}=1$, is compact in $C^{\infty }$ topology.
\end{enumerate}
\end{theorem}

Note the positivity of $Y_{4}\left( g\right) $ is equivalent to the
positivity of Paneitz operator $P$. There are several criterion for the
positivity of $P$ (see \cite[Theorem 1.6]{CHY} and \cite{GM, XY1}). On the
other hand, in a recent preprint \cite{GHL}, it is proved that if $\left(
M,g\right) $ is a smooth compact Riemannian manifold with dimension $n\geq 6$%
, and $Y\left( g\right) >0$, $Y_{4}\left( g\right) >0$, then we can find a
conformal metric $\widetilde{g}$ with $\widetilde{R}>0$ and $\widetilde{Q}>0$%
. In particular, it follows from \cite{GM} that any conformal metric with
constant $Q$ curvature must have positive scalar curvature. Similar
statement for $n=5$ is likely to be true but could not be justified due to
the approach there.

In general it is not known whether $Y\left( g\right) >0,Q\geq 0$ and not
identically zero would imply $Y_{4}\left( g\right) >0$. To get around this
difficulty when proving Theorem \ref{thm1.1} we note that by \cite[%
Proposition 1.1]{HY4} if $Y\left( g\right) >0$, $Q\geq 0$ and not
identically zero then $\ker P=0$, and the Green's function of $P$, $G_{P}>0$%
. Hence we can define an integral operator (the inverse of $P$) as%
\begin{equation}
G_{P}f\left( p\right) =\int_{M}G_{P}\left( p,q\right) f\left( q\right) d\mu
\left( q\right) .  \label{eq1.14}
\end{equation}%
If we denote $f=\rho ^{\frac{n+4}{n-4}}$, then equation (\ref{eq1.12})
becomes%
\begin{equation}
G_{P}f=\frac{2}{n-4}f^{\frac{n-4}{n+4}},\quad f\in C^{\infty }\left(
M\right) ,f>0.  \label{eq1.15}
\end{equation}%
Let%
\begin{eqnarray}
\Theta _{4}\left( g\right) &=&\sup_{f\in L^{\frac{2n}{n+4}}\left( M\right)
\backslash \left\{ 0\right\} }\frac{\int_{M}G_{P}f\cdot fd\mu }{\left\Vert
f\right\Vert _{L^{\frac{2n}{n+4}}}^{2}}  \label{eq1.16} \\
&=&\sup_{f\in L^{\frac{2n}{n+4}}\left( M\right) \backslash \left\{ 0\right\}
}\frac{\int_{M\times M}G_{P}\left( p,q\right) f\left( p\right) f\left(
q\right) d\mu \left( p\right) d\mu \left( q\right) }{\left\Vert f\right\Vert
_{L^{\frac{2n}{n+4}}}^{2}}.  \notag
\end{eqnarray}%
It follows from the classical Hardy-Littlewood-Sobolev inequality (\cite{St}%
) that $\Theta _{4}\left( g\right) $ is always finite. Moreover it follows
from (\ref{eq1.9}) that for a positive smooth function $\rho $, $\Theta
_{4}\left( \rho ^{\frac{4}{n-4}}g\right) =\Theta _{4}\left( g\right) $ i.e. $%
\Theta _{4}\left( g\right) $ is a conformal invariant. If $\Theta _{4}\left(
g\right) $ is achieved by a maximizer $f$, using the fact that $G_{P}>0$, we
easily deduce that $f$ cannot change sign. $\Theta _{4}\left( g\right) $ has
a nice geometric description (see Lemma \ref{lem2.1}):%
\begin{equation}
\Theta _{4}\left( g\right) =\frac{2}{n-4}\sup \left\{ \frac{\int_{M}%
\widetilde{Q}d\widetilde{\mu }}{\left\Vert \widetilde{Q}\right\Vert _{L^{%
\frac{2n}{n+4}}\left( M,d\widetilde{\mu }\right) }^{2}}:\widetilde{g}\in %
\left[ g\right] \right\}  \label{eq1.17}
\end{equation}%
Here $\left[ g\right] $ denotes the conformal class of $g$ i.e.%
\begin{equation}
\left[ g\right] =\left\{ \rho ^{2}g:\rho \in C^{\infty }\left( M\right)
,\rho >0\right\} .  \label{eq1.18}
\end{equation}

\begin{theorem}
\label{thm1.3}Assume $\left( M,g\right) $ is a smooth compact $n$
dimensional Riemannian manifold with $n\geq 5$, $Y\left( g\right) >0$, $%
Q\geq 0$ and not identically zero, then

\begin{enumerate}
\item $\Theta _{4}\left( g\right) \geq \Theta _{4}\left( S^{n}\right) $,
here $S^{n}$ has the standard metric. $\Theta _{4}\left( g\right) =\Theta
_{4}\left( S^{n}\right) $ if and only if $\left( M,g\right) $ is conformally
diffeomorphic to the standard sphere.

\item $\Theta _{4}\left( g\right) $ is always achieved. Any maximizer $f$
must be smooth and cannot change sign. If $f>0$, then after scaling we have $%
G_{P}f=\frac{2}{n-4}f^{\frac{n-4}{n+4}}$ i.e. $Q_{f^{\frac{4}{n+4}}g}=1$.

\item If $\left( M,g\right) $ is not conformally diffeomorphic to the
standard sphere, then the set of all maximizers $f$ for $\Theta _{4}\left(
g\right) $, after normalizing with $\left\Vert f\right\Vert _{L^{\frac{2n}{%
n+4}}}=1$, is compact in the $C^{\infty }$ topology.
\end{enumerate}
\end{theorem}

It is worthwhile to note the similarity of Theorem \ref{thm1.2} and \ref%
{thm1.3} to classical Yamabe problem (\cite{LP,S}) and the integral equation
considered in \cite{HWY1,HWY2}. Indeed, the formulation of our approach
follows that of \cite{HWY2}. Integral equation formulation of the $Q$
curvature equation was used in \cite{QR2}. A similar functional for the
conformal Laplacian operator, $\Theta _{2}$ (see (\ref{eq4.8})) is also
considered in \cite{DoZ}. In Section \ref{sec2} below we will first give
other expressions for $\Theta _{4}\left( g\right) $ and discuss its relation
with $Y_{4}\left( g\right) $, then we will derive the concentration
compactness principle for the extremal problem of $\Theta _{4}\left(
g\right) $ and find the asymptotic expansion formula for the Green's
function of Paneitz operator. In Section \ref{sec3} we will show that
maximizers always exist and that they are smooth. In particular Theorem \ref%
{thm1.3} will follow. At last, in Section \ref{sec4} we will prove Theorem %
\ref{thm1.2}. Moreover we will show the approach to Theorem \ref{thm1.3}
gives another way to find constant scalar curvature metrics in a conformal
class.

The authors would like to thank Gursky and Malchiodi for making their work
available. We would also like to thank the referee for his/her careful
reading of the article and many comments which improve the presentation of
the paper.

\section{Some preparations\label{sec2}}

\subsection{The conformal invariants $Y_{4}\left( g\right) ,Y_{4}^{+}\left(
g\right) $ and $\Theta _{4}\left( g\right) $\label{sec2.1}}

Throughout this subsection we will assume $\left( M,g\right) $ is a smooth
compact $n$ dimensional Riemannian manifold with $n\geq 5$. Recall that%
\begin{equation}
Y_{4}\left( g\right) =\inf_{u\in H^{2}\left( M\right) \backslash \left\{
0\right\} }\frac{E\left( u\right) }{\left\Vert u\right\Vert _{L^{\frac{2n}{%
n-4}}}^{2}}=\inf_{u\in C^{\infty }\left( M\right) \backslash \left\{
0\right\} }\frac{\int_{M}Pu\cdot ud\mu }{\left\Vert u\right\Vert _{L^{\frac{%
2n}{n-4}}}^{2}}.  \label{eq2.1}
\end{equation}%
If in addition $Y\left( g\right) >0$, $Q\geq 0$ and not identically zero,
then%
\begin{eqnarray}
\Theta _{4}\left( g\right) &=&\sup_{f\in L^{\frac{2n}{n+4}}\left( M\right)
\backslash \left\{ 0\right\} }\frac{\int_{M}G_{P}f\cdot fd\mu }{\left\Vert
f\right\Vert _{L^{\frac{2n}{n+4}}}^{2}}  \label{eq2.2} \\
&=&\sup_{u\in W^{4,\frac{2n}{n+4}}\left( M\right) \backslash \left\{
0\right\} }\frac{\int_{M}Pu\cdot ud\mu }{\left\Vert Pu\right\Vert _{L^{\frac{%
2n}{n+4}}}^{2}}.  \notag
\end{eqnarray}%
The second equality in (\ref{eq2.2}) will be very useful for us later on
because the expression is local. It will facilitate our calculations in
estimating $\Theta _{4}\left( g\right) $. $\Theta _{4}\left( g\right) $ also
has a geometric description.

\begin{lemma}
\label{lem2.1}If $n\geq 5,Y\left( g\right) >0,Q\geq 0$ and not identically
zero, then%
\begin{equation}
\Theta _{4}\left( g\right) =\frac{2}{n-4}\sup \left\{ \frac{\int_{M}%
\widetilde{Q}d\widetilde{\mu }}{\left\Vert \widetilde{Q}\right\Vert _{L^{%
\frac{2n}{n+4}}\left( M,d\widetilde{\mu }\right) }^{2}}:\widetilde{g}\in %
\left[ g\right] \right\} .  \label{eq2.3}
\end{equation}
\end{lemma}

\begin{proof}
Note that%
\begin{eqnarray*}
&&\frac{2}{n-4}\sup \left\{ \frac{\int_{M}\widetilde{Q}d\widetilde{\mu }}{%
\left\Vert \widetilde{Q}\right\Vert _{L^{\frac{2n}{n+4}}\left( M,d\widetilde{%
\mu }\right) }^{2}}:\widetilde{g}\in \left[ g\right] \right\} \\
&=&\sup \left\{ \frac{\int_{M}Pu\cdot ud\mu }{\left\Vert Pu\right\Vert _{L^{%
\frac{2n}{n+4}}}^{2}}:u\in C^{\infty }\left( M\right) ,u>0\right\} \\
&\leq &\Theta _{4}\left( g\right) .
\end{eqnarray*}%
On the other hand, by the positivity of $G_{P}$ we have%
\begin{eqnarray*}
&&\Theta _{4}\left( g\right) \\
&=&\sup \left\{ \frac{\int_{M}G_{P}f\cdot fd\mu }{\left\Vert f\right\Vert
_{L^{\frac{2n}{n+4}}}^{2}}:f\in L^{\frac{2n}{n+4}}\left( M\right) \backslash
\left\{ 0\right\} ,f\geq 0\right\} \\
&=&\sup \left\{ \frac{\int_{M}G_{P}f\cdot fd\mu }{\left\Vert f\right\Vert
_{L^{\frac{2n}{n+4}}}^{2}}:f\in C^{\infty }\left( M\right) \backslash
\left\{ 0\right\} ,f\geq 0\right\} \\
&=&\sup \left\{ \frac{\int_{M}Pu\cdot ud\mu }{\left\Vert Pu\right\Vert _{L^{%
\frac{2n}{n+4}}}^{2}}:u\in C^{\infty }\left( M\right) \backslash \left\{
0\right\} ,Pu\geq 0\right\} \\
&\leq &\sup \left\{ \frac{\int_{M}Pu\cdot ud\mu }{\left\Vert Pu\right\Vert
_{L^{\frac{2n}{n+4}}}^{2}}:u\in C^{\infty }\left( M\right) ,u>0\right\} \\
&=&\frac{2}{n-4}\sup \left\{ \frac{\int_{M}\widetilde{Q}d\widetilde{\mu }}{%
\left\Vert \widetilde{Q}\right\Vert _{L^{\frac{2n}{n+4}}\left( M,d\widetilde{%
\mu }\right) }^{2}}:\widetilde{g}\in \left[ g\right] \right\} .
\end{eqnarray*}%
In between we have used the fact for smooth function $u$, $Pu\geq 0$ and $u$
not identically zero implies $u>0$.
\end{proof}

To better understand the relation between $Y_{4}\left( g\right) $ and $%
\Theta _{4}\left( g\right) $, we define%
\begin{eqnarray}
Y_{4}^{+}\left( g\right) &=&\inf \left\{ \frac{\int_{M}Pu\cdot ud\mu }{%
\left\Vert u\right\Vert _{L^{\frac{2n}{n-4}}}^{2}}:u\in C^{\infty }\left(
M\right) ,u>0\right\}  \label{eq2.4} \\
&=&\frac{n-4}{2}\inf \left\{ \frac{\int_{M}\widetilde{Q}d\widetilde{\mu }}{%
\left( \widetilde{\mu }\left( M\right) \right) ^{\frac{n-4}{n}}}:\widetilde{g%
}\in \left[ g\right] \right\} .  \notag
\end{eqnarray}%
Clearly we have%
\begin{equation}
Y_{4}\left( g\right) \leq Y_{4}^{+}\left( g\right) .  \label{eq2.5}
\end{equation}

\begin{lemma}
\label{lem2.2}If $n\geq 5$, $Y\left( g\right) >0$, $Q\geq 0$ and not
identically zero, then%
\begin{equation}
Y_{4}^{+}\left( g\right) \Theta _{4}\left( g\right) \leq 1.  \label{eq2.6}
\end{equation}%
Moreover if $Y_{4}^{+}\left( g\right) $ is achieved, then $Y_{4}^{+}\left(
g\right) \Theta _{4}\left( g\right) =1$ and $\Theta _{4}\left( g\right) $
must be achieved too.
\end{lemma}

\begin{proof}
It is clear that $\Theta _{4}\left( g\right) >0$. To prove the inequality we
only need to deal with the case $Y_{4}^{+}\left( g\right) >0$. Under this
assumption for $u\in C^{\infty }\left( M\right) ,u>0$, we have $%
\int_{M}Pu\cdot ud\mu >0$. By Holder's inequality we have%
\begin{equation*}
\frac{\left( \int_{M}Pu\cdot ud\mu \right) ^{2}}{\left\Vert u\right\Vert
_{L^{\frac{2n}{n-4}}}^{2}\left\Vert Pu\right\Vert _{L^{\frac{2n}{n+4}}}^{2}}%
\leq 1.
\end{equation*}%
It follows that%
\begin{equation*}
Y_{4}^{+}\left( g\right) \frac{\int_{M}Pu\cdot ud\mu }{\left\Vert
Pu\right\Vert _{L^{\frac{2n}{n+4}}}^{2}}\leq 1.
\end{equation*}%
By the proof of Lemma \ref{lem2.1} we have%
\begin{equation*}
\Theta _{4}\left( g\right) =\sup \left\{ \frac{\int_{M}Pv\cdot vd\mu }{%
\left\Vert Pv\right\Vert _{L^{\frac{2n}{n+4}}}^{2}}:v\in C^{\infty }\left(
M\right) ,v>0\right\} ,
\end{equation*}%
hence $Y_{4}^{+}\left( g\right) \Theta _{4}\left( g\right) \leq 1$.

If $Y_{4}^{+}\left( g\right) $ is achieved, say at $u\in C^{\infty }\left(
M\right) ,u>0$, then%
\begin{equation*}
Pu=\kappa u^{\frac{n+4}{n-4}}
\end{equation*}%
for some constant $\kappa $. Since $G_{P}>0$, we see that $\kappa >0$. Hence%
\begin{equation*}
\Theta _{4}\left( g\right) \geq \frac{\int_{M}Pu\cdot ud\mu }{\left\Vert
Pu\right\Vert _{L^{\frac{2n}{n+4}}}^{2}}=\frac{1}{\kappa }\left\Vert
u\right\Vert _{L^{\frac{2n}{n+4}}}^{-\frac{8}{n-4}}=\frac{1}{Y_{4}^{+}\left(
g\right) }\geq \Theta _{4}\left( g\right) .
\end{equation*}%
Hence all the inequalities are equalities. $\Theta _{4}\left( g\right) =%
\frac{1}{Y_{4}^{+}\left( g\right) }$ and it is achieved at $u$ too.
\end{proof}

\begin{remark}
\label{rmk2.1}Assume $Y_{4}^{+}\left( g\right) \Theta _{4}\left( g\right) =1$%
. Later we will show that $\Theta _{4}\left( g\right) $ is always achieved
by positive smooth functions i.e.%
\begin{equation*}
\Theta _{4}\left( g\right) =\frac{\int_{M}G_{P}f\cdot fd\mu }{\left\Vert
f\right\Vert _{L^{\frac{2n}{n+4}}}^{2}}=\frac{\int_{M}Pv\cdot vd\mu }{%
\left\Vert Pv\right\Vert _{L^{\frac{2n}{n+4}}}^{2}},
\end{equation*}%
here $f\in C^{\infty }\left( M\right) ,f>0$, $v=G_{P}f$. Hence $v\in
C^{\infty }\left( M\right) ,v>0$ and%
\begin{equation*}
Pv=\kappa v^{\frac{n+4}{n-4}}
\end{equation*}%
for some constant $\kappa $. Using $G_{P}>0$ we see that $\kappa >0$. On the
other hand%
\begin{equation*}
\Theta _{4}\left( g\right) =\frac{\int_{M}Pv\cdot vd\mu }{\left\Vert
Pv\right\Vert _{L^{\frac{2n}{n+4}}}^{2}}=\kappa ^{-1}\left\Vert v\right\Vert
_{L^{\frac{2n}{n-4}}}^{-\frac{8}{n-4}}.
\end{equation*}%
Hence%
\begin{equation*}
Y_{4}^{+}\left( g\right) =\kappa \left\Vert v\right\Vert _{L^{\frac{2n}{n-4}%
}}^{\frac{8}{n-4}}=\frac{\int_{M}Pv\cdot vd\mu }{\left\Vert v\right\Vert
_{L^{\frac{2n}{n-4}}}^{2}}.
\end{equation*}%
In other words, positive maximizers for $\Theta _{4}\left( g\right) $ are
also minimizers for $Y_{4}^{+}\left( g\right) $.
\end{remark}

\subsection{The sphere $S^{n}$\label{sec2.2}}

On $S^{n}$ ($n\geq 5$) with standard metric we have%
\begin{equation}
Q=\frac{n\left( n+2\right) \left( n-2\right) }{8}  \label{eq2.7}
\end{equation}%
and%
\begin{equation}
Pu=\Delta ^{2}u-\frac{n^{2}-2n-4}{2}\Delta u+\frac{n\left( n+2\right) \left(
n-2\right) \left( n-4\right) }{16}u.  \label{eq2.8}
\end{equation}%
Let $N$ be the north pole and $\pi _{N}:S^{n}\backslash \left\{ N\right\}
\rightarrow \mathbb{R}^{n}$ be the stereographic projection. Using $x=\pi
_{N}$ as the coordinate, then the Green's function of $P$ with pole at $N$
is given by%
\begin{equation}
G_{P,N}=\frac{1}{n\left( n-2\right) \left( n-4\right) 2^{n-3}\omega _{n}}%
\left( \left\vert x\right\vert ^{2}+1\right) ^{\frac{n-4}{2}}.  \label{eq2.9}
\end{equation}%
Here $\omega _{n}$ is the volume of the unit ball in $\mathbb{R}^{n}$ i.e.%
\begin{equation}
\omega _{n}=\frac{\pi ^{\frac{n}{2}}}{\Gamma \left( \frac{n}{2}+1\right) },
\label{eq2.10}
\end{equation}%
$\Gamma $ is the Gamma function given by%
\begin{equation}
\Gamma \left( \alpha \right) =\int_{0}^{\infty }e^{-t}t^{\alpha -1}dt\quad 
\text{for }\alpha >0.  \label{eq2.11}
\end{equation}%
From \cite{CnLO,Li} we know%
\begin{eqnarray}
Y_{4}\left( S^{n}\right) &=&\inf_{u\in C_{c}^{\infty }\left( \mathbb{R}%
^{n}\right) \backslash \left\{ 0\right\} }\frac{\left\Vert \Delta
u\right\Vert _{L^{2}\left( \mathbb{R}^{n}\right) }^{2}}{\left\Vert
u\right\Vert _{L^{\frac{2n}{n-4}}\left( \mathbb{R}^{n}\right) }^{2}}
\label{eq2.12} \\
&=&\frac{\left\Vert \Delta u_{1}\right\Vert _{L^{2}\left( \mathbb{R}%
^{n}\right) }^{2}}{\left\Vert u_{1}\right\Vert _{L^{\frac{2n}{n-4}}\left( 
\mathbb{R}^{n}\right) }^{2}}  \notag \\
&=&\frac{n\left( n+2\right) \left( n-2\right) \left( n-4\right) }{16}\frac{%
2^{\frac{4}{n}}\pi ^{\frac{2\left( n+1\right) }{n}}}{\Gamma \left( \frac{n+1%
}{2}\right) ^{\frac{4}{n}}}  \notag \\
&=&Y_{4}^{+}\left( S^{n}\right) .  \notag
\end{eqnarray}%
Here%
\begin{equation}
u_{1}\left( x\right) =\left( \left\vert x\right\vert ^{2}+1\right) ^{-\frac{%
n-4}{2}}.  \label{eq2.13}
\end{equation}%
For $\lambda >0$, let%
\begin{equation}
u_{\lambda }\left( x\right) =\lambda ^{-\frac{n-4}{2}}u_{1}\left( \frac{x}{%
\lambda }\right) =\left( \frac{\lambda }{\left\vert x\right\vert
^{2}+\lambda ^{2}}\right) ^{\frac{n-4}{2}},  \label{eq2.14}
\end{equation}%
then%
\begin{equation}
\Delta ^{2}u_{\lambda }=n\left( n+2\right) \left( n-2\right) \left(
n-4\right) u_{\lambda }^{\frac{n+4}{n-4}}.  \label{eq2.15}
\end{equation}

On the other hand it follows from \cite{CnLO,Li} that%
\begin{eqnarray}
&&\Theta _{4}\left( S^{n}\right)  \label{eq2.16} \\
&=&\frac{1}{2n\left( n-2\right) \left( n-4\right) \omega _{n}}\sup_{f\in
L^{2}\left( \mathbb{R}^{n}\right) \backslash \left\{ 0\right\} }\frac{\int_{%
\mathbb{R}^{n}\times \mathbb{R}^{n}}\frac{f\left( x\right) f\left( y\right) 
}{\left\vert x-y\right\vert ^{n-4}}dxdy}{\left\Vert f\right\Vert _{L^{\frac{%
2n}{n+4}}\left( \mathbb{R}^{n}\right) }^{2}}  \notag \\
&=&\sup_{u\in C_{c}^{\infty }\left( \mathbb{R}^{n}\right) \backslash \left\{
0\right\} }\frac{\int_{\mathbb{R}^{n}}\left( \Delta u\right) ^{2}dx}{%
\left\Vert \Delta ^{2}u\right\Vert _{L^{\frac{2n}{n+4}}\left( \mathbb{R}%
^{n}\right) }^{2}}  \notag \\
&=&\frac{1}{2n\left( n-2\right) \left( n-4\right) \omega _{n}}\frac{\int_{%
\mathbb{R}^{n}\times \mathbb{R}^{n}}\frac{f_{1}\left( x\right) f_{1}\left(
y\right) }{\left\vert x-y\right\vert ^{n-4}}dxdy}{\left\Vert
f_{1}\right\Vert _{L^{\frac{2n}{n+4}}\left( \mathbb{R}^{n}\right) }^{2}} 
\notag \\
&=&\frac{1}{Y_{4}\left( S^{n}\right) }.  \notag
\end{eqnarray}%
Here%
\begin{equation}
f_{1}\left( x\right) =\left( \left\vert x\right\vert ^{2}+1\right) ^{-\frac{%
n+4}{2}}.  \label{eq2.17}
\end{equation}%
For $\lambda >0$, let%
\begin{equation}
f_{\lambda }\left( x\right) =\lambda ^{-\frac{n+4}{2}}f_{1}\left( \frac{x}{%
\lambda }\right) =\left( \frac{\lambda }{\left\vert x\right\vert
^{2}+\lambda ^{2}}\right) ^{\frac{n+4}{2}},  \label{eq2.18}
\end{equation}%
then%
\begin{equation}
\Delta ^{2}u_{\lambda }=n\left( n+2\right) \left( n-2\right) \left(
n-4\right) f_{\lambda }.  \label{eq2.19}
\end{equation}

\subsection{Concentration compactness principle\label{sec2.3}}

Here we apply the concentration compactness principle in \cite{Ln} to
extremal problem (\ref{eq2.2}). To achieve this goal we start with an almost
sharp Sobolev inequality. Recall by (\ref{eq2.16}) for $u\in C_{c}^{\infty
}\left( \mathbb{R}^{n}\right) ,$%
\begin{equation}
\int_{\mathbb{R}^{n}}\left( \Delta u\right) ^{2}dx\leq \Theta _{4}\left(
S^{n}\right) \left\Vert \Delta ^{2}u\right\Vert _{L^{\frac{2n}{n+4}}\left( 
\mathbb{R}^{n}\right) }^{2}.  \label{eq2.20}
\end{equation}

\begin{lemma}
\label{lem2.3}Assume $M$ is a smooth compact Riemannian manifold with
dimension $n\geq 5$. Then for any $\varepsilon >0$, we have%
\begin{equation}
\left\Vert \Delta u\right\Vert _{L^{2}\left( M\right) }^{2}\leq \left(
\Theta _{4}\left( S^{n}\right) +\varepsilon \right) \left\Vert Pu\right\Vert
_{L^{\frac{2n}{n+4}}\left( M\right) }^{2}+C\left( \varepsilon \right)
\left\Vert u\right\Vert _{L^{\frac{2n}{n+4}}\left( M\right) }^{2}
\label{eq2.21}
\end{equation}%
for all $u\in W^{4,\frac{2n}{n+4}}\left( M\right) $.
\end{lemma}

The passage from (\ref{eq2.20}) to (\ref{eq2.21}) is standard and we refer
the readers to \cite{DHL,He} for further details. The above almost sharp
Sobolev inequality can be used to prove the following concentration
compactness lemma. We refer the readers to \cite{He,Ln} for the now standard
argument.

\begin{lemma}
\label{lem2.4}Let $M$ be a smooth compact Riemannian manifold with dimension 
$n\geq 5$, $\ker P=0$, $f_{i}\in L^{\frac{2n}{n+4}}\left( M\right) $ such
that $f_{i}\rightharpoonup f$ weakly in $L^{\frac{2n}{n+4}}$. Let $%
u_{i},u\in W^{4,\frac{2n}{n+4}}\left( M\right) $ such that $%
Pu_{i}=f_{i},Pu=f $. Assume%
\begin{equation*}
\left\vert f_{i}\right\vert ^{\frac{2n}{n+4}}d\mu \rightharpoonup \sigma 
\text{ in }\mathcal{M}\left( M\right)
\end{equation*}%
and%
\begin{equation*}
\left\vert \Delta u_{i}\right\vert ^{2}d\mu \rightharpoonup \nu \text{ in }%
\mathcal{M}\left( M\right) ,
\end{equation*}%
here $\mathcal{M}\left( M\right) $ is the space of all Radon measures on $M$%
. Then there exists countably many points $p_{i}\in M$ such that%
\begin{equation*}
\sigma \geq \left\vert f\right\vert ^{\frac{2n}{n+4}}d\mu +\sum_{i}\sigma
_{i}\delta _{p_{i}}
\end{equation*}%
and%
\begin{equation*}
\nu =\left\vert \Delta u\right\vert ^{2}d\mu +\sum_{i}\nu _{i}\delta
_{p_{i}},
\end{equation*}%
here $\sigma _{i}=\sigma \left( \left\{ p_{i}\right\} \right) ,\nu _{i}=\nu
\left( \left\{ p_{i}\right\} \right) $. Moreover%
\begin{equation*}
\nu _{i}\leq \Theta _{4}\left( S^{n}\right) \sigma _{i}^{\frac{n+4}{n}}.
\end{equation*}
\end{lemma}

Now we are ready to derive a criterion for the existence of maximizers. Such
kind of criterion is an analog statement for those of Yamabe problems (\cite%
{LP}).

\begin{proposition}
\label{prop2.1}Assume $\left( M,g\right) $ is a smooth compact $n$
dimensional Riemannian manifold with $n\geq 5$, $\ker P=0$. Let%
\begin{equation*}
\Theta _{4}\left( g\right) =\sup_{f\in L^{\frac{2n}{n+4}}\left( M\right)
\backslash \left\{ 0\right\} }\frac{\int_{M}G_{P}f\cdot fd\mu }{\left\Vert
f\right\Vert _{L^{\frac{2n}{n+4}}}^{2}}.
\end{equation*}%
If $\Theta _{4}\left( g\right) >\Theta _{4}\left( S^{n}\right) $ and $%
f_{i}\in L^{\frac{2n}{n+4}}$ satisfies $\left\Vert f_{i}\right\Vert _{L^{%
\frac{2n}{n+4}}}=1,\int_{M}G_{P}f_{i}\cdot f_{i}d\mu \rightarrow \Theta
_{4}\left( g\right) $, then after passing to a subsequence, we can find a $%
f\in L^{\frac{2n}{n+4}}$ such that $f_{i}\rightarrow f$ in $L^{\frac{2n}{n+4}%
}$. In particular, $\left\Vert f\right\Vert _{L^{\frac{2n}{n+4}}}=1$ and $%
\int_{M}G_{P}f\cdot fd\mu =\Theta _{4}\left( g\right) $, $f$ is a maximizer
for $\Theta _{4}\left( g\right) $.
\end{proposition}

\begin{proof}
After passing to a subsequence we can assume $f_{i}\rightharpoonup f$ weakly
in $L^{\frac{2n}{n+4}}$. Let $u_{i},u\in W^{4,\frac{2n}{n+4}}$ such that $%
Pu_{i}=f_{i},Pu=f$. Then $u_{i}\rightharpoonup u$ weakly in $W^{4,\frac{2n}{%
n+4}}$, $u_{i}\rightarrow u$ in $W^{3,\frac{2n}{n+4}}$ and $u_{i}\rightarrow
u$ in $W^{1,2}$. After passing to another subsequence we have%
\begin{equation*}
\left\vert f_{i}\right\vert ^{\frac{2n}{n+4}}d\mu \rightharpoonup d\sigma 
\text{ and }\left( \Delta u_{i}\right) ^{2}d\mu \rightharpoonup d\nu \text{
in }\mathcal{M}\left( M\right) ,
\end{equation*}%
moreover it follows from Lemma \ref{lem2.4} that%
\begin{equation*}
\sigma \geq \left\vert f\right\vert ^{\frac{2n}{n+4}}d\mu +\sum_{i}\sigma
_{i}\delta _{p_{i}},\quad \nu =\left( \Delta u\right) ^{2}d\mu +\sum_{i}\nu
_{i}\delta _{p_{i}},
\end{equation*}%
here $\sigma _{i}=\sigma \left( \left\{ p_{i}\right\} \right) ,\nu _{i}=\nu
\left( \left\{ p_{i}\right\} \right) $ and%
\begin{equation*}
\nu _{i}\leq \Theta _{4}\left( S^{n}\right) \sigma _{i}^{\frac{n+4}{n}}.
\end{equation*}%
It follows that $\sigma \left( M\right) =1$ and%
\begin{eqnarray*}
&&\int_{M}G_{P}f_{i}\cdot f_{i}d\mu \\
&=&\int_{M}u_{i}Pu_{i}d\mu =E\left( u_{i}\right) \\
&=&\int_{M}\left( \left( \Delta u_{i}\right) ^{2}-4A\left( \nabla
u_{i},\nabla u_{i}\right) +\left( n-2\right) J\left\vert \nabla
u_{i}\right\vert ^{2}+\frac{n-4}{2}Qu_{i}^{2}\right) d\mu \\
&\rightarrow &E\left( u\right) +\sum_{i}\nu _{i}.
\end{eqnarray*}%
Hence%
\begin{eqnarray*}
\Theta _{4}\left( g\right) &=&E\left( u\right) +\sum_{i}\nu _{i} \\
&\leq &\Theta _{4}\left( g\right) \left\Vert f\right\Vert _{L^{\frac{2n}{n+4}%
}}^{2}+\Theta _{4}\left( S^{n}\right) \sum_{i}\sigma _{i}^{\frac{n+4}{n}} \\
&\leq &\Theta _{4}\left( g\right) \left[ \left( \left\Vert f\right\Vert _{L^{%
\frac{2n}{n+4}}}^{\frac{2n}{n+4}}\right) ^{\frac{n+4}{n}}+\sum_{i}\sigma
_{i}^{\frac{n+4}{n}}\right] \\
&\leq &\Theta _{4}\left( g\right) \left( \left\Vert f\right\Vert _{L^{\frac{%
2n}{n+4}}}^{\frac{2n}{n+4}}+\sum_{i}\sigma _{i}\right) ^{\frac{n+4}{n}} \\
&\leq &\Theta _{4}\left( g\right) .
\end{eqnarray*}%
Hence all inequalities become equalities. In particular, $\sigma _{i}=0$, $%
\nu _{i}=0$, $\left\Vert f\right\Vert _{L^{\frac{2n}{n+4}}}=1$. Hence $%
f_{i}\rightarrow f$ in $L^{\frac{2n}{n+4}}$, $E\left( u\right)
=\int_{M}G_{P}f\cdot fd\mu =\Theta _{4}\left( g\right) $.
\end{proof}

\subsection{Expansion of Green's function of the Paneitz operator\label%
{sec2.4}}

In \cite{LP}, the expansion formula of Green's function of conformal
Laplacian operator plays an important role. Here we determine the expansion
formulas for Green's function of Paneitz operator. These formulas will be
crucial in the choice of test function in section \ref{sec3}.

We use the same strategy as \cite[section 6]{LP}, but since we need to take
into account lower order terms, some efforts are needed in doing the
algebra. Let us introduce some notation. For $m\in \mathbb{Z}_{+}$, let 
\begin{equation}
\mathcal{P}_{m}=\left\{ \text{homogeneous degree }m\text{ polynomials on }%
\mathbb{R}^{n}\right\} ,  \label{eq2.22}
\end{equation}%
and%
\begin{equation}
\mathcal{H}_{m}=\left\{ \text{harmonic degree }m\text{ homogeneous
polynomials}\right\} .  \label{eq2.23}
\end{equation}

Let $f$ be a function\ defined on a neighborhood of $0$ except at $0$,
namely $U\backslash \left\{ 0\right\} ,$ $m$ be nonnegative integer, and $%
\theta \in \mathbb{R}$. Then we write $f=O^{\left( m\right) }\left(
r^{\theta }\right) $ as $r\rightarrow 0$ if 
\begin{equation}
f\in C^{m}\left( U\backslash \left\{ 0\right\} \right) \text{ and }\partial
_{i_{1}\cdots i_{k}}f\left( x\right) =O\left( r^{\theta -k}\right) \text{ as 
}r\rightarrow 0  \label{eq2.24}
\end{equation}%
for $k=0,1,\cdots ,m$. Here $r=\left\vert x\right\vert $.

Another useful notation is as follows. Let $f$ be a function defined on a
neighborhood of $0$, namely $U$, $m$ and $k$ be nonnegative integers. Then
we write $f=O_{m}\left( r^{k}\right) $ if $f\in C^{m}\left( U\right) $ and $%
f\left( x\right) =O\left( r^{k}\right) $ as $r\rightarrow 0$. Similarly we
write $f=O_{\infty }\left( r^{k}\right) $ if $f\in C^{\infty }\left(
U\right) $ and $f\left( x\right) =O\left( r^{k}\right) $ as $r\rightarrow 0$.

Let $M$ be a smooth compact manifold with a conformal class of Riemannian
metrics. For a point $p\in M$, choose a conformal normal coordinate (see 
\cite{LP}) at $p$, $x_{1},\cdots ,x_{n}$. Let the metric $%
g=g_{ij}dx_{i}dx_{j}$. Then we have%
\begin{eqnarray}
J\left( p\right) &=&0,\quad J_{i}\left( p\right) =0,\quad \Delta J\left(
p\right) =-\frac{\left\vert W\left( p\right) \right\vert ^{2}}{12\left(
n-1\right) },  \label{eq2.25} \\
A_{ij}\left( p\right) &=&0,\quad A_{ijk}\left( p\right) x_{i}x_{j}x_{k}=0,
\label{eq2.26}
\end{eqnarray}%
and%
\begin{equation}
A_{ijkl}\left( p\right) x_{i}x_{j}x_{k}x_{l}=-\frac{2}{9\left( n-2\right) }%
\sum_{kl}\left( W_{ikjl}\left( p\right) x_{i}x_{j}\right) ^{2}-\frac{r^{2}}{%
n-2}J_{ij}\left( p\right) x_{i}x_{j}.  \label{eq2.27}
\end{equation}%
Here $A_{ijk}$ and $A_{ijkl}$ are covariant derivatives of the Schouten
tensor $A$ (see (\ref{eq1.2})).

\begin{proposition}
\label{prop2.2}Assume $n\geq 5$ and ker$P=0$. Then in conformal normal
coordinate at $p$, we have the following statements:

\begin{itemize}
\item If the original conformal class is conformal flat in a neighborhood of 
$p$, then we may choose $g$ such that it is flat near $p$, and 
\begin{equation}
2n\left( 2-n\right) \left( 4-n\right) \omega _{n}G_{P,p}=r^{4-n}+O_{\infty
}\left( 1\right) .  \label{eq2.28}
\end{equation}

\item If $n$ is odd, then 
\begin{equation}
2n\left( 2-n\right) \left( 4-n\right) \omega _{n}G_{P,p}=r^{4-n}\left(
1+\sum_{i=4}^{n}\psi _{i}\right) +O_{4}\left( 1\right) .  \label{eq2.29}
\end{equation}%
Here $\psi _{i}\in \mathcal{P}_{i}$.

\item If $n$ is even and larger than or equal to $8$, then 
\begin{eqnarray}
&&2n\left( 2-n\right) \left( 4-n\right) \omega _{n}G_{P,p}  \label{eq2.30} \\
&=&r^{4-n}\left( 1+\sum_{i=4}^{n}\psi _{i}\right) +r^{4-n}\log
r\sum_{i=n-4}^{n}\psi _{i}^{\prime }+r^{4-n}\log ^{2}r\sum_{i=n-2}^{n}\psi
_{i}^{\prime \prime }  \notag \\
&&+r^{4-n}\log ^{3}r\cdot \psi _{n}^{\prime \prime \prime }+O_{4}\left(
1\right) .  \notag
\end{eqnarray}%
Here $\psi _{i},\psi _{i}^{\prime },\psi _{i}^{\prime \prime },\psi
_{i}^{\prime \prime \prime }\in \mathcal{P}_{i}$.

\item If $n=6$, then 
\begin{eqnarray}
96\omega _{6}G_{P,p} &=&r^{-2}\left( 1+\psi _{4}+\psi _{5}+\psi _{6}\right)
+r^{-2}\log r\left( \psi _{4}^{\prime }+\psi _{5}^{\prime }+\psi
_{6}^{\prime }\right)  \label{eq2.31} \\
&&+r^{-2}\log ^{2}r\cdot \psi _{6}^{\prime \prime }+O_{4}\left( 1\right) . 
\notag
\end{eqnarray}%
Here $\psi _{i},\psi _{i}^{\prime },\psi _{i}^{\prime \prime }\in \mathcal{P}%
_{i}$.
\end{itemize}

As a consequence, we have

\begin{itemize}
\item If $n=5,6,7$ or $M$ is conformal flat near $p$, then 
\begin{equation}
2n\left( 2-n\right) \left( 4-n\right) \omega _{n}G_{P,p}=r^{4-n}+A+O^{\left(
4\right) }\left( r\right) .  \label{eq2.32}
\end{equation}%
Here $A$ is a constant.

\item If $n=8$, then 
\begin{equation}
384\omega _{8}G_{P,p}=r^{-4}-\frac{\left\vert W\left( p\right) \right\vert
^{2}}{1440}\log r+O^{\left( 4\right) }\left( 1\right) .  \label{eq2.33}
\end{equation}

\item If $n\geq 9$, then 
\begin{equation}
2n\left( 2-n\right) \left( 4-n\right) \omega _{n}G_{P,p}=r^{4-n}+r^{4-n}\psi
_{4}+O^{\left( 4\right) }\left( r^{9-n}\right) ,  \label{2.34}
\end{equation}%
here $\psi _{4}\in \mathcal{P}_{4}$ and in fact 
\begin{eqnarray}
&&\psi _{4}  \label{eq2.35} \\
&=&\frac{1}{40\left( n-2\right) }\left[ \frac{2}{9}\sum_{kl}\left(
W_{ikjl}\left( p\right) x_{i}x_{j}\right) ^{2}-\frac{2r^{2}}{9\left(
n+4\right) }\sum_{jkl}\left( W_{ijkl}\left( p\right) x_{i}+W_{ilkj}\left(
p\right) x_{i}\right) ^{2}\right.  \notag \\
&&\left. +\frac{\left\vert W\left( p\right) \right\vert ^{2}}{3\left(
n+2\right) \left( n+4\right) }r^{4}\right] +\frac{r^{2}}{48\left( n-6\right) 
}\left[ \frac{4}{9\left( n+4\right) }\sum_{jkl}\left( W_{ijkl}\left(
p\right) x_{i}+W_{ilkj}\left( p\right) x_{i}\right) ^{2}\right.  \notag \\
&&\left. -2\left( n-6\right) J_{ij}\left( p\right) x_{i}x_{j}-\frac{\left(
n^{2}+6n-32\right) \left\vert W\left( p\right) \right\vert ^{2}}{6n\left(
n+4\right) \left( n-1\right) }r^{2}\right]  \notag \\
&&+r^{4}\cdot \frac{\left( n-4\right) \left( 3n^{2}-2n-64\right) \left\vert
W\left( p\right) \right\vert ^{2}}{576n\left( n+2\right) \left( n-1\right)
\left( n-6\right) \left( n-8\right) }.  \notag
\end{eqnarray}%
The terms in the square brackets are harmonic polynomials.
\end{itemize}
\end{proposition}

To derive these expansions, we need some algebraic preparations. Note that $%
\mathcal{P}_{m}$ has the following decomposition (see \cite{S})%
\begin{equation}
\mathcal{P}_{m}=\bigoplus_{k=0}^{\left[ \frac{m}{2}\right] }\left( r^{2k}%
\mathcal{H}_{m-2k}\right) .  \label{eq2.36}
\end{equation}%
Under this decomposition, we have 
\begin{equation}
\left. \left( r^{2}\Delta \right) \right\vert _{r^{2k}\mathcal{H}%
_{m-2k}}=2k\left( 2m-2k+n-2\right) \quad \text{for }k=0,1,2,\cdots ,\left[ 
\frac{m}{2}\right] .  \label{eq2.37}
\end{equation}%
Here $\Delta $ denotes the Laplace operator with respect to the Euclidean
metric.

For $\alpha \in \mathbb{R}$, let 
\begin{equation}
A_{\alpha }=r^{2}\Delta +2\alpha r\partial _{r}+\alpha \left( \alpha
+n-2\right) ,  \label{eq2.38}
\end{equation}%
and%
\begin{equation}
B_{\alpha }=\frac{\partial }{\partial \alpha }A_{\alpha }=2r\partial
_{r}+\left( 2\alpha +n-2\right) ,  \label{eq2.39}
\end{equation}%
then 
\begin{eqnarray*}
\Delta \left( r^{\alpha }\varphi \right) &=&r^{\alpha -2}A_{\alpha }\varphi ,
\\
A_{\alpha }\left( r^{\beta }\varphi \right) &=&r^{\beta }A_{\alpha +\beta
}\varphi , \\
A_{\alpha }\left( \varphi \log r\right) &=&\left( A_{\alpha }\varphi \right)
\log r+B_{\alpha }\varphi , \\
B_{\alpha }\left( r^{\beta }\varphi \right) &=&r^{\beta }B_{\alpha +\beta
}\varphi , \\
B_{\alpha }\left( \varphi \log r\right) &=&\left( B_{\alpha }\varphi \right)
\log r+2\varphi .
\end{eqnarray*}%
In addition, 
\begin{eqnarray}
\left. A_{\alpha }\right\vert _{\mathcal{P}_{m}} &=&r^{2}\Delta +\alpha
\left( 2m+\alpha +n-2\right) ,  \label{eq2.40} \\
\left. B_{\alpha }\right\vert _{\mathcal{P}_{m}} &=&2m+2\alpha +n-2,
\label{eq2.41}
\end{eqnarray}%
and 
\begin{equation}
\left. A_{\alpha }\right\vert _{r^{2k}\mathcal{H}_{m-2k}}=\left( \alpha
+2k\right) \left( 2m-2k+\alpha +n-2\right)  \label{eq2.42}
\end{equation}%
for $k=0,1,2,\cdots ,\left[ \frac{m}{2}\right] $. In particular,%
\begin{eqnarray}
&&\left. \left( A_{2-n}A_{4-n}\right) \right\vert _{r^{2k}\mathcal{H}_{m-2k}}
\label{eq2.43} \\
&=&\left( 2m-2k\right) \left( 2m-2k+2\right) \left( 2k+2-n\right) \left(
2k+4-n\right) ,  \notag
\end{eqnarray}%
for $k=0,1,2,\cdots ,\left[ \frac{m}{2}\right] $.

\begin{lemma}
\label{lem2.5}For any real numbers $\alpha $ and $\beta $, and any
nonnegative integer $k$, we have 
\begin{eqnarray*}
B_{\alpha }\left( \varphi \log ^{k}r\right) &=&B_{\alpha }\varphi \cdot \log
^{k}r+2k\varphi \log ^{k-1}r, \\
A_{\alpha }\left( \varphi \log ^{k}r\right) &=&A_{\alpha }\varphi \cdot \log
^{k}r+kB_{\alpha }\varphi \cdot \log ^{k-1}r+k\left( k-1\right) \varphi \log
^{k-2}r,
\end{eqnarray*}%
and 
\begin{eqnarray*}
&&A_{\alpha }A_{\beta }\left( \varphi \log ^{k}r\right) \\
&=&A_{\alpha }A_{\beta }\varphi \cdot \log ^{k}r+k\left( A_{\alpha }B_{\beta
}\varphi +B_{\alpha }A_{\beta }\varphi \right) \log ^{k-1}r \\
&&+k\left( k-1\right) \left( A_{\alpha }\varphi +A_{\beta }\varphi
+B_{\alpha }B_{\beta }\varphi \right) \log ^{k-2}r \\
&&+k\left( k-1\right) \left( k-2\right) \left( B_{\alpha }\varphi +B_{\beta
}\varphi \right) \log ^{k-3}r \\
&&+k\left( k-1\right) \left( k-2\right) \left( k-3\right) \varphi \log
^{k-4}r.
\end{eqnarray*}
\end{lemma}

\begin{proof}
Observe%
\begin{equation*}
\frac{\partial }{\partial \alpha }B_{\alpha }\varphi =2\varphi ,\quad \frac{%
\partial ^{2}}{\partial \alpha ^{2}}B_{\alpha }\varphi =0.
\end{equation*}%
Now since $B_{\alpha }\left( r^{\beta }\varphi \right) =r^{\beta }B_{\alpha
+\beta }\varphi $, we know 
\begin{eqnarray*}
B_{\alpha }\left( \varphi \log ^{k}r\right) &=&\left. \frac{\partial ^{k}}{%
\partial \beta ^{k}}\right\vert _{\beta =0}B_{\alpha }\left( r^{\beta
}\varphi \right) \\
&=&\left. \frac{\partial ^{k}}{\partial \beta ^{k}}\right\vert _{\beta
=0}\left( r^{\beta }B_{\alpha +\beta }\varphi \right) \\
&=&B_{\alpha }\varphi \cdot \log ^{k}r+2k\varphi \log ^{k-1}r,
\end{eqnarray*}%
here we have used the Newton-Lebniz formula. For the second equation, we
start with 
\begin{equation*}
\frac{\partial }{\partial \alpha }A_{\alpha }\varphi =B_{\alpha }\varphi
,\quad \frac{\partial ^{2}}{\partial \alpha ^{2}}A_{\alpha }\varphi
=2\varphi ,\quad \frac{\partial ^{3}}{\partial \alpha ^{3}}A_{\alpha
}\varphi =0,
\end{equation*}%
then 
\begin{eqnarray*}
A_{\alpha }\left( \varphi \log ^{k}r\right) &=&\left. \frac{\partial ^{k}}{%
\partial \beta ^{k}}\right\vert _{\beta =0}A_{\alpha }\left( r^{\beta
}\varphi \right) \\
&=&\left. \frac{\partial ^{k}}{\partial \beta ^{k}}\right\vert _{\beta
=0}\left( r^{\beta }A_{\alpha +\beta }\varphi \right) \\
&=&A_{\alpha }\varphi \cdot \log ^{k}r+kB_{\alpha }\varphi \cdot \log
^{k-1}r+k\left( k-1\right) \varphi \log ^{k-2}r.
\end{eqnarray*}
\end{proof}

Define an operator 
\begin{equation}
M_{g}\varphi =4\func{div}\left( A\left( \nabla _{g}\varphi ,e_{i}\right)
e_{i}\right) +\left( 2-n\right) \func{div}\left( J\nabla _{g}\varphi \right)
.  \label{eq2.44}
\end{equation}%
The Paneitz operator can be written as 
\begin{equation}
P_{g}\varphi =\Delta _{g}^{2}\varphi +M_{g}\varphi +\frac{n-4}{2}Q\varphi .
\label{eq2.45}
\end{equation}%
For any $\alpha \in \mathbb{R}$, define 
\begin{eqnarray}
N_{\alpha ,g}\varphi &=&r^{4}M_{g}\varphi +8\alpha r^{2}A\left( r\partial
_{r},\nabla _{g}\varphi \right) +2\left( 2-n\right) \alpha r^{2}J\cdot
r\partial _{r}\varphi  \label{eq2.46} \\
&&+4\alpha r^{2}\func{div}\left( A\left( r\partial _{r},e_{i}\right)
e_{i}\right) \varphi +\left( 2-n\right) \alpha r^{2}\cdot r\partial
_{r}J\cdot \varphi  \notag \\
&&+4\alpha \left( \alpha -2\right) A\left( r\partial _{r},r\partial
_{r}\right) \varphi +\left( 2-n\right) \alpha \left( \alpha +n-2\right)
r^{2}J\varphi ,  \notag
\end{eqnarray}%
then 
\begin{equation}
M_{g}\left( r^{\alpha }\varphi \right) =r^{\alpha -4}N_{\alpha ,g}\varphi .
\label{eq2.47}
\end{equation}

At first, we claim that 
\begin{equation}
P_{g}\left( r^{4-n}\right) =2n\left( 2-n\right) \left( 4-n\right) \omega
_{n}\delta _{p}+fr^{-n},  \label{eq2.48}
\end{equation}%
with $f=O_{\infty }\left( r^{4}\right) $.

Indeed, since $r^{4-n}$ is radial, we have 
\begin{equation}
\Delta _{g}^{2}\left( r^{4-n}\right) =2n\left( 2-n\right) \left( 4-n\right)
\omega _{n}\delta _{p}.  \label{eq2.49}
\end{equation}%
On the other hand,%
\begin{equation}
M_{g}\left( r^{4-n}\right) =r^{-n}N_{4-n,g}1.  \notag
\end{equation}%
In view of the facts 
\begin{eqnarray*}
&&\func{div}\left( A\left( r\partial _{r},e_{i}\right) e_{i}\right) \\
&=&\partial _{k}\left( x_{i}A_{ij}g^{jk}\right) \\
&=&g^{ij}A_{ij}+x_{i}\partial _{k}A_{ij}g^{jk}+O_{\infty }\left( r^{2}\right)
\\
&=&J+x_{i}A_{ijk}\left( p\right) \delta _{jk}+O_{\infty }\left( r^{2}\right)
\\
&=&x_{i}J_{i}\left( p\right) +O_{\infty }\left( r^{2}\right) \\
&=&O_{\infty }\left( r^{2}\right) ,
\end{eqnarray*}%
and 
\begin{equation*}
A\left( r\partial _{r},r\partial _{r}\right) =A_{ij}x_{i}x_{j}=A_{ijk}\left(
p\right) x_{i}x_{j}x_{k}+O_{\infty }\left( r^{4}\right) =O_{\infty }\left(
r^{4}\right) ,
\end{equation*}%
we see $N_{4-n,g}1\in O_{\infty }\left( r^{4}\right) $, (\ref{eq2.48})
follows.

To continue, first we introduce a notation. For any $\alpha \in \mathbb{R}$,
let 
\begin{equation}
A_{\alpha ,g}=r^{2}\Delta _{g}+2\alpha r\partial _{r}+\alpha \left( \alpha
+n-2\right) ,  \label{eq2.50}
\end{equation}%
then 
\begin{eqnarray*}
\Delta _{g}\left( r^{\alpha }\varphi \right) &=&r^{\alpha -2}A_{\alpha
,g}\varphi , \\
A_{\alpha ,g}\left( r^{\beta }\varphi \right) &=&r^{\beta }A_{\alpha +\beta
,g}\varphi , \\
A_{\alpha ,g}\left( \varphi \log r\right) &=&A_{\alpha ,g}\varphi \cdot \log
r+B_{\alpha }\varphi .
\end{eqnarray*}%
Note that 
\begin{equation}
A_{\alpha ,g}=A_{\alpha }+r^{2}\left( \Delta _{g}-\Delta \right) =A_{\alpha
}+r^{2}\partial _{i}\left( \left( g^{ij}-\delta _{ij}\right) \partial
_{j}\right) .  \label{eq2.51}
\end{equation}

A straightforward computation shows%
\begin{equation}
P_{g}\left( r^{\alpha }\varphi \right) =r^{\alpha -4}\left( A_{\alpha
-2}A_{\alpha }\varphi +K_{\alpha }\varphi \right) ,  \label{eq2.52}
\end{equation}%
where 
\begin{eqnarray}
&&K_{\alpha }\varphi  \label{eq2.53} \\
&=&A_{\alpha -2}\left( r^{2}\left( \Delta _{g}-\Delta \right) \varphi
\right) +r^{2}\left( \Delta _{g}-\Delta \right) A_{\alpha ,g}\varphi
+N_{\alpha ,g}\varphi +\frac{n-4}{2}r^{4}Q\varphi .  \notag
\end{eqnarray}%
We easily see that for any nonnegative integer $k$, $\varphi =O_{\infty
}\left( r^{k}\right) $ implies $K_{\alpha }\varphi =O_{\infty }\left(
r^{k+2}\right) $.

We also introduce the following two operators,%
\begin{eqnarray}
K_{\alpha }^{\left( 1\right) }\varphi &=&\frac{\partial }{\partial \alpha }%
K_{\alpha }\varphi  \label{eq2.54} \\
&=&B_{\alpha -2}\left( r^{2}\left( \Delta _{g}-\Delta \right) \varphi
\right) +r^{2}\left( \Delta _{g}-\Delta \right) B_{\alpha }\varphi  \notag \\
&&+8r^{2}A\left( r\partial _{r},\nabla _{g}\varphi \right) +2\left(
2-n\right) r^{2}J\cdot r\partial _{r}\varphi  \notag \\
&&+4r^{2}\func{div}\left( A\left( r\partial _{r},e_{i}\right) e_{i}\right)
\varphi +\left( 2-n\right) r^{2}\cdot r\partial _{r}J\cdot \varphi  \notag \\
&&+8\left( \alpha -1\right) A\left( r\partial _{r},r\partial _{r}\right)
\varphi +\left( 2-n\right) \left( 2\alpha +n-2\right) r^{2}J\varphi ,  \notag
\end{eqnarray}%
and 
\begin{eqnarray}
K_{\alpha }^{\left( 2\right) }\varphi &=&\frac{\partial }{\partial \alpha }%
K_{\alpha }^{\left( 1\right) }\varphi  \label{eq2.55} \\
&=&4r^{2}\left( \Delta _{g}-\Delta \right) \varphi +8A\left( r\partial
_{r},r\partial _{r}\right) \varphi +2\left( 2-n\right) r^{2}J\varphi  \notag
\\
&=&K^{\left( 2\right) }\varphi .  \notag
\end{eqnarray}%
Clearly, $\varphi =O_{\infty }\left( r^{k}\right) $ for some nonnegative
integer would imply $K_{\alpha }^{\left( 1\right) }\varphi ,K^{\left(
2\right) }\varphi =O_{\infty }\left( r^{k+2}\right) $. In addition, these
operators satisfy the following 
\begin{eqnarray*}
K_{\alpha }\left( r^{\beta }\varphi \right) &=&r^{\beta }K_{\alpha +\beta
}\varphi , \\
K_{\alpha }\left( \varphi \log r\right) &=&K_{\alpha }\varphi \cdot \log
r+K_{\alpha }^{\left( 1\right) }\varphi , \\
K_{\alpha }^{\left( 1\right) }\left( r^{\beta }\varphi \right) &=&r^{\beta
}K_{\alpha +\beta }^{\left( 1\right) }\varphi , \\
K_{\alpha }^{\left( 1\right) }\left( \varphi \log r\right) &=&K_{\alpha
}^{\left( 1\right) }\varphi \cdot \log r+K_{\alpha }^{\left( 2\right)
}\varphi , \\
K^{\left( 2\right) }\left( r^{\beta }\varphi \right) &=&r^{\beta }K^{\left(
2\right) }\varphi , \\
K^{\left( 2\right) }\left( \varphi \log r\right) &=&K^{\left( 2\right)
}\varphi \cdot \log r.
\end{eqnarray*}%
More generally, we have

\begin{lemma}
\label{lem2.6}For any nonnegative integer $k$, we have 
\begin{eqnarray*}
K_{\alpha }^{\left( 1\right) }\left( \varphi \log ^{k}r\right) &=&K_{\alpha
}^{\left( 1\right) }\varphi \cdot \log ^{k}r+kK^{\left( 2\right) }\varphi
\cdot \log ^{k-1}r, \\
K_{\alpha }\left( \varphi \log ^{k}r\right) &=&K_{\alpha }\varphi \cdot \log
^{k}r+kK_{\alpha }^{\left( 1\right) }\varphi \cdot \log ^{k-1}r+\frac{%
k\left( k-1\right) }{2}K^{\left( 2\right) }\varphi \cdot \log ^{k-2}\varphi .
\end{eqnarray*}
\end{lemma}

This follows from the same proof of Lemma \ref{lem2.4}.

\begin{case}
\label{case2.1}The dimension $n$ is odd.
\end{case}

In this case, we claim that we may find a $\psi =\sum_{i=1}^{n}\psi _{i}$,
with $\psi _{i}\in \mathcal{P}_{i}$ such that 
\begin{equation}
A_{2-n}A_{4-n}\psi +K_{4-n}\psi +f=O_{\infty }\left( r^{n+1}\right) .
\label{eq2.56}
\end{equation}%
Once this has been done, then we have 
\begin{equation*}
r^{-n}\left( A_{2-n}A_{4-n}\psi +K_{4-n}\psi +f\right) \in C^{\alpha }\quad 
\text{for any }0<\alpha <1\text{.}
\end{equation*}%
If the domain is small enough, then we may find $\overline{\psi }\in
C^{4,\alpha }$ such that 
\begin{equation*}
P_{g}\overline{\psi }=-r^{-n}\left( A_{2-n}A_{4-n}\psi +K_{4-n}\psi
+f\right) .
\end{equation*}%
Then 
\begin{equation}
P_{g}\left( r^{4-n}\left( 1+\psi \right) +\overline{\psi }\right) =2n\left(
2-n\right) \left( 4-n\right) \omega _{n}\delta _{p}.  \label{eq2.57}
\end{equation}%
Hence the Green's function satisfies 
\begin{equation}
2n\left( 2-n\right) \left( 4-n\right) \omega _{n}G_{p}=r^{4-n}\left( 1+\psi
\right) +\overline{\psi }+O_{\infty }\left( 1\right) .  \label{eq2.58}
\end{equation}

To define $\psi _{1},\cdots ,\psi _{n}$, we let $\psi _{1}=0,\psi _{2}=0$
and $\psi _{3}=0$. One easily see 
\begin{eqnarray}
f_{3} &=&A_{2-n}A_{4-n}\left( \psi _{1}+\psi _{2}+\psi _{3}\right)
+K_{4-n}\left( \psi _{1}+\psi _{2}+\psi _{3}\right) +f  \label{eq2.59} \\
&=&f=O_{\infty }\left( r^{4}\right) .  \notag
\end{eqnarray}%
Assume we have found $\psi _{1},\psi _{2},\cdots ,\psi _{k}$ for $3\leq
k\leq n-1$, such that $\psi _{i}\in \mathcal{P}_{i}$ and 
\begin{equation*}
f_{k}=A_{2-n}A_{4-n}\left( \sum_{i=1}^{k}\psi _{i}\right) +K_{4-n}\left(
\sum_{i=1}^{k}\psi _{i}\right) +f=O_{\infty }\left( r^{k+1}\right) ,
\end{equation*}%
then we write $f_{k}=\phi _{k+1}+O_{\infty }\left( r^{k+2}\right) ,\phi
_{k+1}\in \mathcal{P}_{k+1}$. Since 
\begin{eqnarray*}
&&\left. A_{2-n}A_{4-n}\right\vert _{r^{2j}\mathcal{H}_{k+1-2j}} \\
&=&\left( 2\left( k+1\right) -2j\right) \left( 2\left( k+1\right)
-2j+2\right) \left( 2j+2-n\right) \left( 2j+4-n\right) \neq 0
\end{eqnarray*}%
for $j=0,1,2,\cdots ,\left[ \frac{k+1}{2}\right] $, $A_{2-n}A_{4-n}$ is
invertible on $\mathcal{P}_{k+1}$. We may find a unique $\psi _{k+1}\in 
\mathcal{P}_{k+1}$, such that 
\begin{equation}
A_{2-n}A_{4-n}\psi _{k+1}+\phi _{k+1}=0.  \label{eq2.60}
\end{equation}%
Then 
\begin{eqnarray*}
f_{k+1} &=&A_{2-n}A_{4-n}\left( \sum_{i=1}^{k+1}\psi _{i}\right)
+K_{4-n}\left( \sum_{i=1}^{k+1}\psi _{i}\right) +f \\
&=&f_{k}+A_{2-n}A_{4-n}\psi _{k+1}+K_{4-n}\psi _{k+1}=O_{\infty }\left(
r^{k+2}\right) .
\end{eqnarray*}%
This finishes the induction process.

\begin{case}
\label{case2.2}$n$ is even and larger than or equal to $8$.
\end{case}

In this case, we first set $\psi _{1}=0,\psi _{2}=0$ and $\psi _{3}=0$.
Since $A_{2-n}A_{4-n}$ is invertible on $\mathcal{P}_{k}$ for $0\leq k\leq
n-5$, by the same induction procedure as Case \ref{case2.1}, we can find $%
\psi _{4},\cdots ,\psi _{n-5}$ such that $\psi _{i}\in \mathcal{P}_{i}$ and 
\begin{equation*}
f_{n-5}=A_{2-n}A_{4-n}\left( \sum_{i=1}^{n-5}\psi _{i}\right) +K_{4-n}\left(
\sum_{i=1}^{n-5}\psi _{i}\right) +f=O_{\infty }\left( r^{n-4}\right) .
\end{equation*}%
To continue, we write 
\begin{equation*}
f_{n-5}=\phi _{n-4}+O_{\infty }\left( r^{n-3}\right) ,\quad \phi _{n-4}\in 
\mathcal{P}_{n-4}.
\end{equation*}%
Let $\psi _{n-4}^{\left( 0\right) }=\alpha _{n-4}^{\left( 0\right) }+\beta
_{n-4}^{\left( 0\right) }\log r$ with $\alpha _{n-4}^{\left( 0\right)
},\beta _{n-4}^{\left( 0\right) }\in \mathcal{P}_{n-4}$, then%
\begin{eqnarray*}
&&A_{2-n}A_{4-n}\psi _{n-4}^{\left( 0\right) } \\
&=&A_{2-n}A_{4-n}\alpha _{n-4}^{\left( 0\right) }+\left(
A_{2-n}B_{4-n}+B_{2-n}A_{4-n}\right) \beta _{n-4}^{\left( 0\right)
}+A_{2-n}A_{4-n}\beta _{n-4}^{\left( 0\right) }\cdot \log r.
\end{eqnarray*}%
Let $\beta _{n-4}^{\left( 0\right) }\in r^{n-4}\mathcal{H}_{0}$, then since 
\begin{equation*}
\left. \left( A_{2-n}B_{4-n}+B_{2-n}A_{4-n}\right) \right\vert _{r^{n-4}%
\mathcal{H}_{0}}=-2\left( n-2\right) \left( n-4\right) \neq 0,
\end{equation*}%
and 
\begin{eqnarray*}
&&\left. A_{2-n}A_{4-n}\right\vert _{r^{2k}\mathcal{H}_{n-4-2k}} \\
&=&\left( 2\left( n-4\right) -2k\right) \left( 2\left( n-4\right)
-2k+2\right) \left( 2k+2-n\right) \left( 2k+4-n\right) \neq 0,
\end{eqnarray*}%
for $0\leq k\leq \frac{n}{2}-3$, we may find a $\alpha _{n-4}^{\left(
0\right) }\in \mathcal{P}_{n-4}$ and a $\beta _{n-4}^{\left( 0\right) }\in
r^{n-4}\mathcal{H}_{0}$ such that%
\begin{equation*}
A_{2-n}A_{4-n}\psi _{n-4}^{\left( 0\right) }+\phi _{n-4}=0.
\end{equation*}%
This implies 
\begin{eqnarray*}
&&f_{n-4} \\
&=&A_{2-n}A_{4-n}\left( \sum_{i=1}^{n-5}\psi _{i}+\psi _{n-4}^{\left(
0\right) }\right) +K_{4-n}\left( \sum_{i=1}^{n-5}\psi _{i}+\psi
_{n-4}^{\left( 0\right) }\right) +f \\
&=&f_{n-5}+A_{2-n}A_{4-n}\psi _{n-4}^{\left( 0\right) }+K_{4-n}\psi
_{n-4}^{\left( 0\right) } \\
&=&O_{\infty }\left( r^{n-3}\right) +O_{\infty }\left( r^{n-2}\right) \log r.
\end{eqnarray*}%
Next we write 
\begin{equation*}
f_{n-4}=\phi _{n-3}+O_{\infty }\left( r^{n-2}\right) \log r+O_{\infty
}\left( r^{n-2}\right) ,\quad \phi _{n-3}\in \mathcal{P}_{n-3}.
\end{equation*}%
Again by similar arguments, we can find $\psi _{n-3}^{\left( 0\right) }\in 
\mathcal{P}_{n-3}+r^{n-4}\mathcal{H}_{1}\log r$ such that%
\begin{equation*}
A_{2-n}A_{4-n}\psi _{n-3}^{\left( 0\right) }+\phi _{n-3}=0.
\end{equation*}%
Then%
\begin{eqnarray*}
&&f_{n-3} \\
&=&A_{2-n}A_{4-n}\left( \sum_{i=1}^{n-5}\psi _{i}+\psi _{n-4}^{\left(
0\right) }+\psi _{n-3}^{\left( 0\right) }\right) +K_{4-n}\left(
\sum_{i=1}^{n-5}\psi _{i}+\psi _{n-4}^{\left( 0\right) }+\psi _{n-3}^{\left(
0\right) }\right) +f \\
&=&f_{n-4}+A_{2-n}A_{4-n}\psi _{n-3}^{\left( 0\right) }+K_{4-n}\psi
_{n-3}^{\left( 0\right) } \\
&=&O_{\infty }\left( r^{n-2}\right) \log r+O_{\infty }\left( r^{n-2}\right) .
\end{eqnarray*}%
We write 
\begin{equation*}
f_{n-3}=\phi _{n-2}^{\left( 1\right) }\log r+O_{\infty }\left(
r^{n-2}\right) +O_{\infty }\left( r^{n-1}\right) \log r.
\end{equation*}%
Similar as before, we may find%
\begin{equation*}
\psi _{n-2}^{\left( 1\right) }\in \mathcal{P}_{n-2}\log r+\left( r^{n-2}%
\mathcal{H}_{0}+r^{n-4}\mathcal{H}_{2}\right) \log ^{2}r
\end{equation*}%
such that%
\begin{equation*}
A_{2-n}A_{4-n}\psi _{n-2}^{\left( 1\right) }+\phi _{n-2}^{\left( 1\right)
}\log r\in \mathcal{P}_{n-2}.
\end{equation*}%
Indeed, for $\psi _{n-2}^{\left( 1\right) }=\alpha _{n-2}^{\left( 1\right)
}\log r+\beta _{n-2}^{\left( 1\right) }\log ^{2}r$, with $\alpha
_{n-2}^{\left( 1\right) },\beta _{n-2}^{\left( 1\right) }\in \mathcal{P}%
_{n-2}$, we have 
\begin{eqnarray*}
&&A_{2-n}A_{4-n}\psi _{n-2}^{\left( 1\right) } \\
&=&\left( A_{2-n}A_{4-n}\alpha _{n-2}^{\left( 1\right) }+2\left(
A_{2-n}B_{4-n}+B_{2-n}A_{4-n}\right) \beta _{n-2}^{\left( 1\right) }\right)
\log r \\
&&+A_{2-n}A_{4-n}\beta _{n-2}^{\left( 1\right) }\cdot \log ^{2}r+\mathcal{P}%
_{n-2}.
\end{eqnarray*}%
Let $\beta _{n-2}^{\left( 1\right) }\in r^{n-2}\mathcal{H}_{0}+r^{n-4}%
\mathcal{H}_{2}$. Since%
\begin{eqnarray*}
\left. 2\left( A_{2-n}B_{4-n}+B_{2-n}A_{4-n}\right) \right\vert _{r^{n-2}%
\mathcal{H}_{0}} &=&4n\left( n-2\right) \neq 0, \\
\left. 2\left( A_{2-n}B_{4-n}+B_{2-n}A_{4-n}\right) \right\vert _{r^{n-4}%
\mathcal{H}_{2}} &=&-4n\left( n+2\right) \neq 0,
\end{eqnarray*}%
and%
\begin{eqnarray*}
&&\left. A_{2-n}A_{4-n}\right\vert _{r^{2k}\mathcal{H}_{n-2-2k}} \\
&=&\left( 2\left( n-2\right) -2k\right) \left( 2\left( n-2\right)
-2k+2\right) \left( 2k+2-n\right) \left( 2k+4-n\right) \neq 0
\end{eqnarray*}%
for $0\leq k\leq \frac{n}{2}-3$, we may find the above needed $\psi
_{n-2}^{\left( 1\right) }$. Then 
\begin{eqnarray*}
&&f_{n-2}^{\left( 1\right) } \\
&=&A_{2-n}A_{4-n}\left( \sum_{i=1}^{n-5}\psi _{i}+\psi _{n-4}^{\left(
0\right) }+\psi _{n-3}^{\left( 0\right) }+\psi _{n-2}^{\left( 1\right)
}\right) \\
&&+K_{4-n}\left( \sum_{i=1}^{n-5}\psi _{i}+\psi _{n-4}^{\left( 0\right)
}+\psi _{n-3}^{\left( 0\right) }+\psi _{n-2}^{\left( 1\right) }\right) +f \\
&=&f_{n-3}+A_{2-n}A_{4-n}\psi _{n-2}^{\left( 1\right) }+K_{4-n}\psi
_{n-2}^{\left( 1\right) } \\
&=&O_{\infty }\left( r^{n-2}\right) +O_{\infty }\left( r^{n-1}\right) \log
r+O_{\infty }\left( r^{n}\right) \log ^{2}r.
\end{eqnarray*}%
The next step is to remove the $\mathcal{P}_{n-2}$ term in $O_{\infty
}\left( r^{n-2}\right) $, then the $\mathcal{P}_{n-1}\log r$ term in $%
O_{\infty }\left( r^{n-1}\right) \log r$ and so on, until we reach $%
O_{\infty }\left( r^{n+1}\right) \log ^{2}r+O_{\infty }\left( r^{n+1}\right)
\log r+O_{\infty }\left( r^{n+1}\right) +O_{\infty }\left( r^{n+2}\right)
\log ^{3}r$. That is, we find%
\begin{eqnarray*}
\psi _{n-4}^{\left( 0\right) } &\in &\mathcal{P}_{n-4}+r^{n-4}\mathcal{H}%
_{0}\log r, \\
\psi _{n-3}^{\left( 0\right) } &\in &\mathcal{P}_{n-3}+r^{n-4}\mathcal{H}%
_{1}\log r, \\
\psi _{n-2}^{\left( 1\right) } &\in &\mathcal{P}_{n-2}\log r+\left( r^{n-2}%
\mathcal{H}_{0}+r^{n-4}\mathcal{H}_{2}\right) \log ^{2}r, \\
\psi _{n-2}^{\left( 0\right) } &\in &\mathcal{P}_{n-2}+\left( r^{n-2}%
\mathcal{H}_{0}+r^{n-4}\mathcal{H}_{2}\right) \log r, \\
\psi _{n-1}^{\left( 1\right) } &\in &\mathcal{P}_{n-1}\log r+\left( r^{n-2}%
\mathcal{H}_{1}+r^{n-4}\mathcal{H}_{3}\right) \log ^{2}r, \\
\psi _{n-1}^{\left( 0\right) } &\in &\mathcal{P}_{n-1}+\left( r^{n-2}%
\mathcal{H}_{1}+r^{n-4}\mathcal{H}_{3}\right) \log r, \\
\psi _{n}^{\left( 2\right) } &\in &\mathcal{P}_{n}\log ^{2}r+\left( r^{n-2}%
\mathcal{H}_{2}+r^{n-4}\mathcal{H}_{4}\right) \log ^{3}r, \\
\psi _{n}^{\left( 1\right) } &\in &\mathcal{P}_{n}\log r+\left( r^{n-2}%
\mathcal{H}_{2}+r^{n-4}\mathcal{H}_{4}\right) \log ^{2}r,
\end{eqnarray*}%
and%
\begin{equation*}
\psi _{n}^{\left( 0\right) }\in \mathcal{P}_{n}+\left( r^{n-2}\mathcal{H}%
_{2}+r^{n-4}\mathcal{H}_{4}\right) \log r,
\end{equation*}%
such that 
\begin{eqnarray*}
f_{n} &=&A_{2-n}A_{4-n}\left( \sum_{i=1}^{n-5}\psi _{i}+\sum_{i=n-4}^{n}\psi
_{i}^{\left( 0\right) }+\sum_{i=n-2}^{n}\psi _{i}^{\left( 1\right) }+\psi
_{n}^{\left( 2\right) }\right) \\
&&+K_{4-n}\left( \sum_{i=1}^{n-5}\psi _{i}+\sum_{i=n-4}^{n}\psi _{i}^{\left(
0\right) }+\sum_{i=n-2}^{n}\psi _{i}^{\left( 1\right) }+\psi _{n}^{\left(
2\right) }\right) +f \\
&=&O_{\infty }\left( r^{n+1}\right) \log ^{2}r+O_{\infty }\left(
r^{n+1}\right) \log r+O_{\infty }\left( r^{n+1}\right) +O_{\infty }\left(
r^{n+2}\right) \log ^{3}r.
\end{eqnarray*}%
Clearly $r^{-n}f_{n}\in C^{\alpha }$ for any $0<\alpha <1$. This implies
locally we may find $\overline{\psi }\in C^{4,\alpha }$ such that $P_{g}%
\overline{\psi }=-r^{-n}f_{n}$. Let 
\begin{equation*}
\psi =\sum_{i=1}^{n-5}\psi _{i}+\sum_{i=n-4}^{n}\psi _{i}^{\left( 0\right)
}+\sum_{i=n-2}^{n}\psi _{i}^{\left( 1\right) }+\psi _{n}^{\left( 2\right) },
\end{equation*}%
then 
\begin{equation*}
P_{g}\left( r^{4-n}\left( 1+\psi \right) +\overline{\psi }\right) =2n\left(
2-n\right) \left( 4-n\right) \omega _{n}\delta _{p}
\end{equation*}%
on a small disk. Hence%
\begin{equation*}
2n\left( 2-n\right) \left( 4-n\right) \omega _{n}G_{p}=r^{4-n}\left( 1+\psi
\right) +\overline{\psi }+O_{\infty }\left( 1\right) .
\end{equation*}

\begin{case}
\label{case2.3}$n=6$.
\end{case}

This case can be done similarly as for Case \ref{case2.2}. That is, we can
find 
\begin{eqnarray*}
\psi _{4}^{\left( 0\right) } &\in &\mathcal{P}_{4}+\left( r^{4}\mathcal{H}%
_{0}+r^{2}\mathcal{H}_{2}\right) \log r, \\
\psi _{5}^{\left( 0\right) } &\in &\mathcal{P}_{5}+\left( r^{4}\mathcal{H}%
_{1}+r^{2}\mathcal{H}_{3}\right) \log r, \\
\psi _{6}^{\left( 1\right) } &\in &\mathcal{P}_{6}\log r+\left( r^{4}%
\mathcal{H}_{2}+r^{2}\mathcal{H}_{4}\right) \log ^{2}r,
\end{eqnarray*}%
and 
\begin{equation*}
\psi _{6}^{\left( 0\right) }\in \mathcal{P}_{6}+\left( r^{4}\mathcal{H}%
_{2}+r^{2}\mathcal{H}_{4}\right) \log r,
\end{equation*}%
such that 
\begin{eqnarray*}
&&f_{6} \\
&=&A_{-4}A_{-2}\left( \psi _{4}^{\left( 0\right) }+\psi _{5}^{\left(
0\right) }+\psi _{6}^{\left( 0\right) }+\psi _{6}^{\left( 1\right) }\right)
+K_{-2}\left( \psi _{4}^{\left( 0\right) }+\psi _{5}^{\left( 0\right) }+\psi
_{6}^{\left( 0\right) }+\psi _{6}^{\left( 1\right) }\right) +f \\
&=&O_{\infty }\left( r^{7}\right) \log r+O_{\infty }\left( r^{7}\right)
+O_{\infty }\left( r^{8}\right) \log ^{2}r.
\end{eqnarray*}%
The remaining argument can be done as before.

\begin{case}
\label{case2.4}$M$ is conformal flat near $p$.
\end{case}

In this case, we may take the metric $g$ such that it is flat near $p$. This
implies $P_{g}=\Delta ^{2}$, and hence 
\begin{equation*}
P_{g}\left( r^{4-n}\right) =2n\left( 2-n\right) \left( 4-n\right) \omega
_{n}\delta _{p}.
\end{equation*}%
It follows that%
\begin{equation*}
2n\left( 2-n\right) \left( 4-n\right) \omega _{n}G_{P,p}=r^{4-n}+O_{\infty
}\left( 1\right) .
\end{equation*}

Finally, to get the leading terms in the expansion for $n\geq 8$, by a
direct computation we have $f_{3}=f=\phi _{4}+O_{\infty }\left( r^{5}\right) 
$, with $\phi _{4}\in \mathcal{P}_{4}$ and 
\begin{eqnarray}
&&\phi _{4}  \label{eq2.61} \\
&=&-\frac{4\left( n-4\right) }{9}\sum_{kl}\left( W_{ikjl}\left( p\right)
x_{i}x_{j}\right) ^{2}+2\left( n-4\right) \left( n-6\right)
r^{2}J_{ij}\left( p\right) x_{i}x_{j}  \notag \\
&&+\frac{\left( n-4\right) \left\vert W\left( p\right) \right\vert ^{2}}{%
24\left( n-1\right) }r^{4}.  \notag
\end{eqnarray}%
From this, we can compute the leading terms of $G_{P,p}$ directly from the
arguments in Case \ref{case2.2}.

\section{Existence and regularity of maximizers\label{sec3}}

The main aim of this section is to show the strict inequality between $%
\Theta _{4}\left( g\right) $ and $\Theta _{4}\left( S^{n}\right) $ in the
assumption of Proposition \ref{prop2.1} is valid as long as $\left(
M,g\right) $ is not conformally equivalent to the standard sphere. As in the
Yamabe problem case (\cite{LP}), this is achieved by a careful choice of
test function.

\begin{proposition}
\label{prop3.1}Assume $\left( M,g\right) $ is a smooth compact $n$
dimensional Riemannian manifold with $n\geq 5$, $Y\left( g\right) >0$, $%
Q\geq 0$ and not identically zero, then%
\begin{equation}
\Theta _{4}\left( g\right) \geq \Theta _{4}\left( S^{n}\right)  \label{eq3.1}
\end{equation}%
and equality holds if and only if $\left( M,g\right) $ is conformally
equivalent to the standard sphere.
\end{proposition}

Before we start the proof of Proposition \ref{prop3.1}, we list several
basic identities which will facilitate the calculations. For $b>-n$ and $%
2a-b>n$,%
\begin{equation}
\int_{\mathbb{R}^{n}}\frac{\left\vert x\right\vert ^{b}}{\left( \left\vert
x\right\vert ^{2}+1\right) ^{a}}dx=\frac{n\omega _{n}}{2}\frac{\Gamma \left( 
\frac{b+n}{2}\right) \Gamma \left( a-\frac{b+n}{2}\right) }{\Gamma \left(
a\right) }=\pi ^{\frac{n}{2}}\frac{\Gamma \left( \frac{b+n}{2}\right) \Gamma
\left( a-\frac{b+n}{2}\right) }{\Gamma \left( a\right) \Gamma \left( \frac{n%
}{2}\right) }.  \label{eq3.2}
\end{equation}%
If we fix an orthonormal frame at $p$, and let $\Delta $ be the Euclidean
Laplacian, then%
\begin{eqnarray}
&&\Delta \sum_{k,l}\left( W_{ikjl}\left( p\right) x_{i}x_{j}\right) ^{2}
\label{eq3.3} \\
&=&4W_{ikjl}\left( p\right) W_{ikml}\left( p\right)
x_{j}x_{m}+4W_{ikjl}\left( p\right) W_{ilmk}\left( p\right) x_{j}x_{m} 
\notag \\
&=&2\sum_{ikl}\left( W_{ikjl}\left( p\right) x_{j}+W_{iljk}\left( p\right)
x_{j}\right) ^{2}  \notag \\
&=&2\sum_{jkl}\left( W_{ijkl}\left( p\right) x_{i}+W_{ilkj}\left( p\right)
x_{i}\right) ^{2},  \notag
\end{eqnarray}%
and%
\begin{equation}
\Delta ^{2}\sum_{kl}\left( W_{ikjl}\left( p\right) x_{i}x_{j}\right)
^{2}=8\left( \left\vert W\left( p\right) \right\vert ^{2}+W_{ikjl}\left(
p\right) W_{iljk}\left( p\right) \right) =12\left\vert W\left( p\right)
\right\vert ^{2},  \label{eq3.4}
\end{equation}%
here we have used%
\begin{equation}
W_{ikjl}\left( p\right) W_{iljk}\left( p\right) =\frac{1}{2}\left\vert
W\left( p\right) \right\vert ^{2},  \label{eq3.5}
\end{equation}%
which follows from the usual Bianchi identity. Hence%
\begin{eqnarray}
&&\sum_{kl}\left( W_{ikjl}\left( p\right) x_{i}x_{j}\right) ^{2}
\label{eq3.6} \\
&=&\left[ \sum_{kl}\left( W_{ikjl}\left( p\right) x_{i}x_{j}\right) ^{2}-%
\frac{r^{2}}{n+4}\sum_{jkl}\left( W_{ijkl}\left( p\right)
x_{i}+W_{ilkj}\left( p\right) x_{i}\right) ^{2}\right.  \notag \\
&&\left. +\frac{3}{2\left( n+2\right) \left( n+4\right) }\left\vert W\left(
p\right) \right\vert ^{2}r^{4}\right] +r^{2}\cdot \left[ \frac{1}{n+4}%
\sum_{jkl}\left( W_{ijkl}\left( p\right) x_{i}+W_{ilkj}\left( p\right)
x_{i}\right) ^{2}\right.  \notag \\
&&\left. -\frac{3}{n\left( n+4\right) }\left\vert W\left( p\right)
\right\vert ^{2}r^{2}\right] +r^{4}\cdot \frac{3}{2n\left( n+2\right) }%
\left\vert W\left( p\right) \right\vert ^{2}.  \notag
\end{eqnarray}%
The polynomials in the square brackets are harmonic. In particular,%
\begin{equation}
\int_{S^{n-1}}\sum_{kl}\left( W_{ikjl}\left( p\right) x_{i}x_{j}\right)
^{2}dS=\frac{3\omega _{n}}{2\left( n+2\right) }\left\vert W\left( p\right)
\right\vert ^{2}.  \label{eq3.7}
\end{equation}

Due to the fact that in (\ref{eq2.2}) the power $\frac{2n}{n+4}<2<\frac{2n}{%
n-4}$, to control the error on annulus region, the choice of test functions
for $\Theta _{4}\left( g\right) $ will be more delicate than those for the
classical Yamabe problem or for $Y_{4}\left( g\right) $ (see (\ref{eq1.13}))
in the literature. In particular, dimension $8$ and $9$ has to be separated
from dimensions $5,6,7$ and dimensions greater than $9$.

Fix a function $\eta _{1}\in C^{\infty }\left( \mathbb{R},\mathbb{R}\right) $
such that $\left. \eta _{1}\right\vert _{\left( -\infty ,1\right) }=0,\left.
\eta _{1}\right\vert _{\left( 2,\infty \right) }=1$ and $0\leq \eta _{1}\leq
1$. Denote $\eta _{2}=1-\eta _{1}$. For convenience we always denote%
\begin{equation}
H=2n\left( 2-n\right) \left( 4-n\right) \omega _{n}G_{P,p}.  \label{eq3.8}
\end{equation}%
Let $\delta $ be a small positive number. For $0<\lambda <\delta $, let%
\begin{equation}
u_{\lambda }=\left( \frac{\lambda }{\left\vert x\right\vert ^{2}+\lambda ^{2}%
}\right) ^{\frac{n-4}{2}}  \label{eq3.9}
\end{equation}%
and%
\begin{equation}
\beta =\lambda ^{\frac{n-4}{2}}r^{4-n}-u_{\lambda }.  \label{eq3.10}
\end{equation}%
If we write $\phi \left( x\right) =\left\vert x\right\vert ^{4-n}-\left(
\left\vert x\right\vert ^{2}+1\right) ^{\frac{4-n}{2}}$, then $\beta
=\lambda ^{\frac{4-n}{2}}\phi \left( \frac{x}{\lambda }\right) $.

\begin{case}
\label{case3.1}$M$ is conformally flat near $p$, $n\geq 5$.
\end{case}

In this case we can assume that the metric $g$ is flat near $p$. Using the
Euclidean coordinate at $p$, namely $x_{1},\cdots ,x_{n}$ we have%
\begin{equation}
H=r^{4-n}+A_{0}+\alpha .  \label{eq3.11}
\end{equation}%
Here $A_{0}$ is a constant, $\alpha =O_{\infty }\left( r\right) $ is a
biharmonic function (with respect to Euclidean metric).

Define%
\begin{equation}
\varphi _{\lambda }=\left\{ 
\begin{array}{ll}
u_{\lambda }+\eta _{1}\left( \frac{r}{\delta }\right) \beta +\lambda ^{\frac{%
n-4}{2}}A_{0}+\lambda ^{\frac{n-4}{2}}\alpha , & \text{on }B_{3\delta
}\left( p\right) , \\ 
\lambda ^{\frac{n-4}{2}}H, & \text{on }M\backslash B_{3\delta }\left(
p\right) .%
\end{array}%
\right.  \label{eq3.12}
\end{equation}%
It is clear that $\varphi _{\lambda }\in C^{\infty }\left( M\right) $. Note
that%
\begin{eqnarray}
&&P\varphi _{\lambda }  \label{eq3.13} \\
&=&\left\{ 
\begin{array}{ll}
n\left( n+2\right) \left( n-2\right) \left( n-4\right) \lambda ^{\frac{n+4}{2%
}}\left( \left\vert x\right\vert ^{2}+\lambda ^{2}\right) ^{-\frac{n+4}{2}},
& \text{on }B_{\delta }\left( p\right) , \\ 
O\left( \lambda ^{\frac{n}{2}}\right) , & \text{on }B_{2\delta }\left(
p\right) \backslash B_{\delta }\left( p\right) , \\ 
0, & \text{on }M\backslash B_{2\delta }\left( p\right) .%
\end{array}%
\right.  \notag
\end{eqnarray}%
Hence%
\begin{eqnarray}
&&\int_{M}\left\vert P\varphi _{\lambda }\right\vert ^{\frac{2n}{n+4}}d\mu
\label{eq3.14} \\
&=&\left( n\left( n+2\right) \left( n-2\right) \left( n-4\right) \right) ^{%
\frac{2n}{n+4}}\frac{\Gamma \left( \frac{n}{2}\right) \pi ^{\frac{n}{2}}}{%
\left( n-1\right) !}+O\left( \lambda ^{\frac{n^{2}}{n+4}}\right) .  \notag
\end{eqnarray}%
It follows that%
\begin{eqnarray}
&&\left\Vert P\varphi _{\lambda }\right\Vert _{L^{\frac{2n}{n+4}}}^{2}
\label{eq3.15} \\
&=&\left( n\left( n+2\right) \left( n-2\right) \left( n-4\right) \right) ^{2}%
\frac{\Gamma \left( \frac{n}{2}\right) ^{\frac{n+4}{n}}\pi ^{\frac{n+4}{2}}}{%
\left( \left( n-1\right) !\right) ^{\frac{n+4}{n}}}+O\left( \lambda ^{\frac{%
n^{2}}{n+4}}\right) .  \notag
\end{eqnarray}%
On the other hand,%
\begin{eqnarray}
&&\int_{M}P\varphi _{\lambda }\cdot \varphi _{\lambda }d\mu  \label{eq3.16}
\\
&=&n\left( n+2\right) \left( n-2\right) \left( n-4\right) \frac{\Gamma
\left( \frac{n}{2}\right) \pi ^{\frac{n}{2}}}{\left( n-1\right) !}+\frac{%
4\left( n-2\right) \left( n-4\right) \pi ^{\frac{n}{2}}}{\Gamma \left( \frac{%
n}{2}\right) }A_{0}\lambda ^{n-4}+o\left( \lambda ^{n-4}\right) .  \notag
\end{eqnarray}%
Hence%
\begin{eqnarray}
&&\frac{\int_{M}P\varphi _{\lambda }\cdot \varphi _{\lambda }d\mu }{%
\left\Vert P\varphi _{\lambda }\right\Vert _{L^{\frac{2n}{n+4}}}^{2}}
\label{eq3.17} \\
&=&\Theta _{4}\left( S^{n}\right) +\frac{4\left( \left( n-1\right) !\right)
^{\frac{n+4}{n}}}{n^{2}\left( n+2\right) ^{2}\left( n-2\right) \left(
n-4\right) \Gamma \left( \frac{n}{2}\right) ^{\frac{2n+4}{n}}\pi ^{2}}%
A_{0}\lambda ^{n-4}+o\left( \lambda ^{n-4}\right) .  \notag
\end{eqnarray}%
If $\left( M,g\right) $ is not conformally diffeomorphic to the standard
sphere, then it follows from the arguments in \cite[section 6]{HY4} that $%
A_{0}>0$. Fix $\delta $ small and let $\lambda \downarrow 0$, we see $\Theta
_{4}\left( g\right) >\Theta _{4}\left( S^{n}\right) $.

\begin{case}
\label{case3.2}$n=5,6,7$.
\end{case}

In this case by a conformal change of the metric we can assume $\exp _{p}$
preserves the volume near $p$ (note this is another way of saying we choose
a conformal normal coordinate, see \cite{LP}). Using the normal coordinate
at $p$, $x_{1},\cdots ,x_{n}$, we have%
\begin{equation}
H=r^{4-n}+A_{0}+\alpha .  \label{eq3.18}
\end{equation}%
Here $A_{0}$ is a constant and $\alpha =O^{\left( 4\right) }\left( r\right) $%
. Define%
\begin{equation}
\varphi _{\lambda }=\left\{ 
\begin{array}{ll}
u_{\lambda }+\eta _{1}\left( \frac{r}{\delta }\right) \beta +\lambda ^{\frac{%
n-4}{2}}A_{0}+\lambda ^{\frac{n-4}{2}}\alpha , & \text{on }B_{3\delta
}\left( p\right) , \\ 
\lambda ^{\frac{n-4}{2}}H, & \text{on }M\backslash B_{3\delta }\left(
p\right) .%
\end{array}%
\right.  \label{eq3.19}
\end{equation}%
then $\varphi _{\lambda }\in W^{4,\frac{2n}{n+4}}\left( M\right) $. On $%
B_{\delta }\left( p\right) \backslash \left\{ p\right\} ,$%
\begin{eqnarray}
&&P\varphi _{\lambda }  \label{eq3.20} \\
&=&Pu_{\lambda }-\lambda ^{\frac{n-4}{2}}P\left( r^{4-n}\right)  \notag \\
&=&\Delta ^{2}u_{\lambda }-4\func{div}\left( A\left( \nabla \beta
,e_{i}\right) e_{i}\right) +\left( n-2\right) \func{div}\left( J\nabla \beta
\right) -\frac{n-4}{2}Q\beta  \notag \\
&=&n\left( n+2\right) \left( n-2\right) \left( n-4\right) \lambda ^{\frac{n+4%
}{2}}\left( \left\vert x\right\vert ^{2}+\lambda ^{2}\right) ^{-\frac{n+4}{2}%
}+O\left( \lambda ^{\frac{n}{2}}\left\vert x\right\vert ^{2-n}\right) . 
\notag
\end{eqnarray}%
Here we will need to use (\ref{eq2.25}) and (\ref{eq2.26}). On $B_{2\delta
}\left( p\right) \backslash B_{\delta }\left( p\right) $,%
\begin{equation}
P\varphi _{\lambda }=-P\left( \eta _{2}\left( \frac{r}{\delta }\right) \beta
\right) =O\left( \lambda ^{\frac{n}{2}}\right)  \label{eq3.21}
\end{equation}%
and on $M\backslash B_{2\delta }\left( p\right) $, $P\varphi _{\lambda }=0$.
Hence%
\begin{eqnarray*}
&&\int_{M}\left\vert P\varphi _{\lambda }\right\vert ^{\frac{2n}{n+4}}d\mu \\
&=&\left( n\left( n+2\right) \left( n-2\right) \left( n-4\right) \right) ^{%
\frac{2n}{n+4}}\frac{\Gamma \left( \frac{n}{2}\right) \pi ^{\frac{n}{2}}}{%
\left( n-1\right) !}+o\left( \lambda ^{n-4}\right) ,
\end{eqnarray*}%
and%
\begin{eqnarray*}
&&\left\Vert P\varphi _{\lambda }\right\Vert _{L^{\frac{2n}{n+4}}}^{2} \\
&=&\left( n\left( n+2\right) \left( n-2\right) \left( n-4\right) \right) ^{2}%
\frac{\Gamma \left( \frac{n}{2}\right) ^{\frac{n+4}{n}}\pi ^{\frac{n+4}{2}}}{%
\left( \left( n-1\right) !\right) ^{\frac{n+4}{n}}}+o\left( \lambda
^{n-4}\right) .
\end{eqnarray*}%
On the other hand,%
\begin{eqnarray}
&&\int_{M}P\varphi _{\lambda }\cdot \varphi _{\lambda }d\mu  \label{eq3.22}
\\
&=&n\left( n+2\right) \left( n-2\right) \left( n-4\right) \frac{\Gamma
\left( \frac{n}{2}\right) \pi ^{\frac{n}{2}}}{\left( n-1\right) !}+\frac{%
4\left( n-2\right) \left( n-4\right) \pi ^{\frac{n}{2}}}{\Gamma \left( \frac{%
n}{2}\right) }A_{0}\lambda ^{n-4}+o\left( \lambda ^{n-4}\right) .  \notag
\end{eqnarray}%
Summing up we have%
\begin{eqnarray}
&&\frac{\int_{M}P\varphi _{\lambda }\cdot \varphi _{\lambda }d\mu }{%
\left\Vert P\varphi _{\lambda }\right\Vert _{L^{\frac{2n}{n+4}}}^{2}}
\label{eq3.23} \\
&=&\Theta _{4}\left( S^{n}\right) +\frac{4\left( \left( n-1\right) !\right)
^{\frac{n+4}{n}}}{n^{2}\left( n+2\right) ^{2}\left( n-2\right) \left(
n-4\right) \Gamma \left( \frac{n}{2}\right) ^{\frac{2n+4}{n}}\pi ^{2}}%
A_{0}\lambda ^{n-4}+o\left( \lambda ^{n-4}\right) .  \notag
\end{eqnarray}%
By \cite[section 6]{HY4} we know when $\left( M,g\right) $ is not
conformally diffeomorphic to the standard sphere, $A_{0}$ is strictly
positive. Letting $\lambda \downarrow 0$, we get $\Theta _{4}\left( g\right)
>\Theta _{4}\left( S^{n}\right) $ in this case.

\begin{case}
\label{case3.3}$\left( M,g\right) $ is not locally conformally flat and $n=8$%
.
\end{case}

In this case we can choose $p$ such that $W\left( p\right) \neq 0$. By a
conformal change of the metric we can assume $\exp _{p}$ preserves the
volume near $p$. Using the normal coordinate at $p$, $x_{1},\cdots ,x_{8}$,
we have%
\begin{equation}
H=r^{-4}-\frac{\left\vert W\left( p\right) \right\vert ^{2}}{1440}\log
r+\alpha .  \label{eq3.24}
\end{equation}%
Here $\alpha =O^{\left( 4\right) }\left( 1\right) $. Define%
\begin{equation}
\varphi _{\lambda }=\left\{ 
\begin{array}{ll}
u_{\lambda }+\eta _{1}\left( \frac{r}{\delta }\right) \beta -\frac{%
\left\vert W\left( p\right) \right\vert ^{2}}{1440}\lambda ^{2}\log
r+\lambda ^{2}\alpha , & \text{on }B_{3\delta }\left( p\right) , \\ 
\lambda ^{2}H, & \text{on }M\backslash B_{3\delta }\left( p\right) .%
\end{array}%
\right.  \label{eq3.25}
\end{equation}%
Then $\varphi \in W^{4,\frac{4}{3}}\left( M\right) $. On $B_{\delta }\left(
p\right) \backslash \left\{ p\right\} ,$%
\begin{eqnarray}
P\varphi _{\lambda } &=&Pu_{\lambda }-\lambda ^{2}P\left( r^{-4}\right)
\label{eq3.26} \\
&=&1920\lambda ^{6}\left( \left\vert x\right\vert ^{2}+\lambda ^{2}\right)
^{-6}-4\func{div}\left( A\left( \nabla \beta ,e_{i}\right) e_{i}\right) +6%
\func{div}\left( J\nabla \beta \right) -2Q\beta  \notag \\
&=&1920\lambda ^{6}\left( \left\vert x\right\vert ^{2}+\lambda ^{2}\right)
^{-6}+O\left( \beta \right) +O\left( \beta ^{\prime }r\right) +O\left( \beta
^{\prime \prime }r^{2}\right) .  \notag
\end{eqnarray}%
Here we have used (\ref{eq2.25}) and (\ref{eq2.26}). On $B_{2\delta }\left(
p\right) \backslash B_{\delta }\left( p\right) $,%
\begin{equation}
P\varphi _{\lambda }=-P\left( \eta _{2}\left( \frac{r}{\delta }\right) \beta
\right) =O\left( \lambda ^{4}\right)  \label{eq3.27}
\end{equation}%
and on $M\backslash B_{2\delta }\left( p\right) $, $P\varphi _{\lambda }=0$.
Note that%
\begin{eqnarray*}
\beta &=&\lambda ^{2}r^{-4}-\lambda ^{2}\left( r^{2}+\lambda ^{2}\right)
^{-2}, \\
\beta ^{\prime } &=&-4\lambda ^{2}r^{-5}+4\lambda ^{2}\left( r^{2}+\lambda
^{2}\right) ^{-3}r, \\
\beta ^{\prime \prime } &=&20\lambda ^{2}r^{-6}-24\lambda ^{2}\left(
r^{2}+\lambda ^{2}\right) ^{-4}r^{2}+4\lambda ^{2}\left( r^{2}+\lambda
^{2}\right) ^{-3}.
\end{eqnarray*}%
Hence we have%
\begin{equation}
\int_{M}\left\vert P\varphi _{\lambda }\right\vert ^{\frac{4}{3}}d\mu =\frac{%
1920^{\frac{4}{3}}\pi ^{4}}{840}+O\left( \lambda ^{4}\right) ,
\label{eq3.28}
\end{equation}%
and%
\begin{equation*}
\int_{M}P\varphi _{\lambda }\cdot \varphi _{\lambda }d\mu =\frac{1920\pi ^{4}%
}{840}+\frac{\pi ^{4}\left\vert W\left( p\right) \right\vert ^{2}}{90}%
\lambda ^{4}\log \frac{1}{\lambda }+O\left( \lambda ^{4}\right) .
\end{equation*}%
It follows that%
\begin{equation*}
\frac{\int_{M}P\varphi _{\lambda }\cdot \varphi _{\lambda }d\mu }{\left\Vert
P\varphi _{\lambda }\right\Vert _{L^{\frac{4}{3}}}^{2}}=\Theta _{4}\left(
S^{8}\right) +\frac{210^{\frac{3}{2}}}{41472000\pi ^{2}}\left\vert W\left(
p\right) \right\vert ^{2}\lambda ^{4}\log \frac{1}{\lambda }+O\left( \lambda
^{4}\right) .
\end{equation*}%
Hence $\Theta _{4}\left( g\right) >\Theta _{4}\left( S^{8}\right) $.

\begin{case}
\label{case3.4}$M$ is not conformally flat and $n=9$.
\end{case}

In this case we can choose $p$ such that $W\left( p\right) \neq 0$. By a
conformal change of metric we can assume $\exp _{p}$ preserves the volume
near $p$. Using the normal coordinate at $p$, $x_{1},\cdots ,x_{9}$, we have%
\begin{equation}
H=r^{-5}+r^{-5}\psi _{4}+\alpha .  \label{eq3.29}
\end{equation}%
Here $\alpha =O^{\left( 4\right) }\left( 1\right) $ and%
\begin{eqnarray}
&&\psi _{4}  \label{eq3.30} \\
&=&\frac{1}{280}\left[ \frac{2}{9}\sum_{kl}\left( W_{ikjl}\left( p\right)
x_{i}x_{j}\right) ^{2}-\frac{2}{117}r^{2}\sum_{jkl}\left( W_{ijkl}\left(
p\right) x_{i}+W_{ilkj}\left( p\right) x_{i}\right) ^{2}\right.  \notag \\
&&\left. +\frac{\left\vert W\left( p\right) \right\vert ^{2}}{429}r^{4}%
\right] +\frac{r^{2}}{144}\left[ \frac{4}{117}\sum_{jkl}\left(
W_{ijkl}\left( p\right) x_{i}+W_{ilkj}\left( p\right) x_{i}\right)
^{2}\right.  \notag \\
&&\left. -6J_{ij}\left( p\right) x_{i}x_{j}-\frac{103}{5616}\left\vert
W\left( p\right) \right\vert ^{2}r^{2}\right] +\frac{805}{1368576}\left\vert
W\left( p\right) \right\vert ^{2}r^{4}.  \notag
\end{eqnarray}%
Define%
\begin{equation}
\varphi _{\lambda }=\left\{ 
\begin{array}{ll}
u_{\lambda }+\eta _{1}\left( \frac{r}{\delta }\right) \beta +\lambda ^{\frac{%
5}{2}}r^{-5}\psi _{4}+\lambda ^{\frac{5}{2}}\alpha , & \text{on }B_{3\delta
}\left( p\right) , \\ 
\lambda ^{\frac{5}{2}}H, & \text{on }M\backslash B_{3\delta }\left( p\right)
.%
\end{array}%
\right.  \label{eq3.31}
\end{equation}%
Then $\varphi \in W^{4,\frac{18}{13}}\left( M\right) $. On $B_{\delta
}\left( p\right) \backslash \left\{ p\right\} ,$%
\begin{eqnarray}
P\varphi _{\lambda } &=&Pu_{\lambda }-\lambda ^{\frac{5}{2}}P\left(
r^{-5}\right)  \label{eq3.32} \\
&=&3465\lambda ^{\frac{13}{2}}\left( \left\vert x\right\vert ^{2}+\lambda
^{2}\right) ^{-\frac{13}{2}}-4\func{div}\left( A\left( \nabla \beta
,e_{i}\right) e_{i}\right) +7\func{div}\left( J\nabla \beta \right) -\frac{5%
}{2}Q\beta  \notag \\
&=&3465\lambda ^{\frac{13}{2}}\left( \left\vert x\right\vert ^{2}+\lambda
^{2}\right) ^{-\frac{13}{2}}-2\left( \frac{\beta ^{\prime }}{r}\right)
^{\prime }\frac{A_{ijkl}\left( p\right) x_{i}x_{j}x_{k}x_{l}}{r}  \notag \\
&&+\frac{7}{2}\left( \frac{\beta ^{\prime }}{r}\right) ^{\prime
}rJ_{ij}\left( p\right) x_{i}x_{j}+\frac{65}{2}\frac{\beta ^{\prime }}{r}%
J_{ij}\left( p\right) x_{i}x_{j}-\frac{5}{192}\left\vert W\left( p\right)
\right\vert ^{2}\beta  \notag \\
&&+O\left( \beta r\right) +O\left( \beta ^{\prime }r^{2}\right) +O\left(
\beta ^{\prime \prime }r^{3}\right) .  \notag
\end{eqnarray}%
On $B_{2\delta }\left( p\right) \backslash B_{\delta }\left( p\right) $,%
\begin{equation}
P\varphi _{\lambda }=-P\left( \eta _{2}\left( \frac{r}{\delta }\right) \beta
\right) =O\left( \lambda ^{\frac{9}{2}}\right)  \label{eq3.33}
\end{equation}%
and on $M\backslash B_{2\delta }\left( p\right) $, $P\varphi _{\lambda }=0$.
Note that%
\begin{eqnarray*}
\beta &=&\lambda ^{\frac{5}{2}}r^{-5}-\lambda ^{\frac{5}{2}}\left(
r^{2}+\lambda ^{2}\right) ^{-\frac{5}{2}}, \\
\beta ^{\prime } &=&-5\lambda ^{\frac{5}{2}}r^{-6}+5\lambda ^{\frac{5}{2}%
}\left( r^{2}+\lambda ^{2}\right) ^{-\frac{7}{2}}r, \\
\frac{\beta ^{\prime }}{r} &=&-5\lambda ^{\frac{5}{2}}r^{-7}+5\lambda ^{%
\frac{5}{2}}\left( r^{2}+\lambda ^{2}\right) ^{-\frac{7}{2}}, \\
\beta ^{\prime \prime } &=&30\lambda ^{\frac{5}{2}}r^{-7}-35\lambda ^{\frac{5%
}{2}}\left( r^{2}+\lambda ^{2}\right) ^{-\frac{9}{2}}r^{2}+5\lambda ^{\frac{5%
}{2}}\left( r^{2}+\lambda ^{2}\right) ^{-\frac{7}{2}}, \\
\left( \frac{\beta ^{\prime }}{r}\right) ^{\prime } &=&35\lambda ^{\frac{5}{2%
}}r^{-8}-35\lambda ^{\frac{5}{2}}\left( r^{2}+\lambda ^{2}\right) ^{-\frac{9%
}{2}}r.
\end{eqnarray*}%
A straightforward calculation shows%
\begin{eqnarray}
&&\int_{M}\left\vert P\varphi _{\lambda }\right\vert ^{\frac{18}{13}}d\mu
\label{eq3.34} \\
&=&\frac{3465^{\frac{18}{13}}\pi ^{5}}{6144}\left[ 1+\left( \frac{94208}{%
4459455}\frac{1}{\pi }-\frac{41}{9009}\right) \left\vert W\left( p\right)
\right\vert ^{2}\lambda ^{4}+o\left( \lambda ^{4}\right) \right] ,  \notag
\end{eqnarray}%
hence%
\begin{eqnarray}
&&\left\Vert P\varphi _{\lambda }\right\Vert _{L^{\frac{18}{13}}}^{2}
\label{eq3.35} \\
&=&\frac{3465^{2}\pi ^{\frac{65}{9}}}{6144^{\frac{13}{9}}}\left[ 1+\left( 
\frac{94208}{3087315}\frac{1}{\pi }-\frac{41}{6237}\right) \left\vert
W\left( p\right) \right\vert ^{2}\lambda ^{4}+o\left( \lambda ^{4}\right) %
\right] .  \notag
\end{eqnarray}%
On the other hand,%
\begin{eqnarray}
&&\int_{M}P\varphi _{\lambda }\cdot \varphi _{\lambda }d\mu  \label{eq3.36}
\\
&=&\frac{1155}{2048}\pi ^{5}\left[ 1+\left( \frac{94208}{3087315}\frac{1}{%
\pi }-\frac{41}{12474}\right) \left\vert W\left( p\right) \right\vert
^{2}\lambda ^{4}+o\left( \lambda ^{4}\right) \right] .  \notag
\end{eqnarray}%
Summing up we get%
\begin{equation}
\frac{\int_{M}P\varphi _{\lambda }\cdot \varphi _{\lambda }d\mu }{\left\Vert
P\varphi _{\lambda }\right\Vert _{L^{\frac{18}{13}}}^{2}}=\Theta _{4}\left(
S^{9}\right) \left( 1+\frac{41}{12474}\left\vert W\left( p\right)
\right\vert ^{2}\lambda ^{4}+o\left( \lambda ^{4}\right) \right) .
\label{eq3.37}
\end{equation}%
Hence we see that $\Theta _{4}\left( g\right) >\Theta _{4}\left(
S^{9}\right) $.

\begin{case}
\label{case3.5}$M$ is not conformally flat and $n\geq 10$.
\end{case}

We can find a point $p$ such that $W\left( p\right) \neq 0$. Let $%
x_{1},\cdots ,x_{n}$ be conformal normal coordinate at $p$, $\delta $ be a
small fixed positive number, and%
\begin{equation}
\varphi _{\lambda }=u_{\lambda }\left( x\right) \eta _{2}\left( \frac{%
\left\vert x\right\vert }{\delta }\right) .  \label{eq3.38}
\end{equation}%
Then on $B_{2\delta }\left( p\right) \backslash B_{\delta }\left( p\right) $,%
\begin{equation}
P\varphi _{\lambda }=O\left( \lambda ^{\frac{n-4}{2}}\right) .
\label{eq3.39}
\end{equation}%
On $B_{\delta }\left( p\right) ,$%
\begin{eqnarray}
&&P\varphi _{\lambda }  \label{eq3.40} \\
&=&n\left( n+2\right) \left( n-2\right) \left( n-4\right) \lambda ^{\frac{n+4%
}{2}}\left( \left\vert x\right\vert ^{2}+\lambda ^{2}\right) ^{-\frac{n+4}{2}%
}  \notag \\
&&-\frac{4}{9}\left( n-4\right) \lambda ^{\frac{n-4}{2}}\left( \left\vert
x\right\vert ^{2}+\lambda ^{2}\right) ^{-\frac{n}{2}}\sum_{kl}\left(
W_{ikjl}\left( p\right) x_{i}x_{j}\right) ^{2}  \notag \\
&&+\frac{n-4}{2}\lambda ^{\frac{n-4}{2}}\left( \left\vert x\right\vert
^{2}+\lambda ^{2}\right) ^{-\frac{n}{2}}\left( 4\left( n-6\right) \left\vert
x\right\vert ^{2}+\left( n^{2}-16\right) \lambda ^{2}\right) J_{ij}\left(
p\right) x_{i}x_{j}  \notag \\
&&+\frac{n-4}{24\left( n-1\right) }\lambda ^{\frac{n-4}{2}}\left\vert
W\left( p\right) \right\vert ^{2}\left( \left\vert x\right\vert ^{2}+\lambda
^{2}\right) ^{-\frac{n-4}{2}}  \notag \\
&&+O\left( \lambda ^{\frac{n-4}{2}}\left( \left\vert x\right\vert
^{2}+\lambda ^{2}\right) ^{-\frac{n-4}{2}}\left\vert x\right\vert \right) . 
\notag
\end{eqnarray}%
Using the basic inequality%
\begin{equation}
\left\vert \left\vert 1+t\right\vert ^{\frac{2n}{n+4}}-1-\frac{2n}{n+4}%
t\right\vert \leq C\left\vert t\right\vert ^{\frac{2n}{n+4}}  \label{eq3.41}
\end{equation}%
we see on $B_{\delta }\left( p\right) ,$%
\begin{eqnarray*}
&&\left\vert P\varphi _{\lambda }\right\vert ^{\frac{2n}{n+4}} \\
&=&\left( n\left( n+2\right) \left( n-2\right) \left( n-4\right) \right) ^{%
\frac{2n}{n+4}}\lambda ^{n}\left( \left\vert x\right\vert ^{2}+\lambda
^{2}\right) ^{-n}\cdot \\
&&\left[ 1-\frac{8}{9}\frac{\lambda ^{-4}\left( \left\vert x\right\vert
^{2}+\lambda ^{2}\right) ^{2}}{\left( n+2\right) \left( n+4\right) \left(
n-2\right) }\sum_{kl}\left( W_{ikjl}\left( p\right) x_{i}x_{j}\right)
^{2}\right. \\
&&+\frac{\lambda ^{-4}\left( \left\vert x\right\vert ^{2}+\lambda
^{2}\right) ^{2}}{\left( n+2\right) \left( n+4\right) \left( n-2\right) }%
\left( 4\left( n-6\right) \left\vert x\right\vert ^{2}+\left(
n^{2}-16\right) \lambda ^{2}\right) J_{ij}\left( p\right) x_{i}x_{j} \\
&&+\frac{\lambda ^{-4}\left\vert W\left( p\right) \right\vert ^{2}}{12\left(
n+2\right) \left( n+4\right) \left( n-1\right) \left( n-2\right) }\left(
\left\vert x\right\vert ^{2}+\lambda ^{2}\right) ^{4} \\
&&+O\left( \lambda ^{-4}\left( \left\vert x\right\vert ^{2}+\lambda
^{2}\right) ^{4}\left\vert x\right\vert \right) +O\left( \lambda ^{-\frac{8n%
}{n+4}}\left( \left\vert x\right\vert ^{2}+\lambda ^{2}\right) ^{\frac{8n}{%
n+4}}\right) \\
&&\left. +O\left( \lambda ^{-\frac{8n}{n+4}}\left( \left\vert x\right\vert
^{2}+\lambda ^{2}\right) ^{\frac{8n}{n+4}}\left\vert x\right\vert ^{\frac{2n%
}{n+4}}\right) \right] .
\end{eqnarray*}%
A straightforward calculation shows%
\begin{eqnarray}
&&\int_{M}\left\vert P\varphi _{\lambda }\right\vert ^{\frac{2n}{n+4}}d\mu
\label{eq3.42} \\
&=&\left( n\left( n+2\right) \left( n-2\right) \left( n-4\right) \right) ^{%
\frac{2n}{n+4}}\frac{\pi ^{\frac{n}{2}}\Gamma \left( \frac{n}{2}\right) }{%
\left( n-1\right) !}\cdot  \notag \\
&&\left( 1-\frac{1}{3}\frac{n^{2}-4n-4}{\left( n+2\right) \left( n+4\right)
\left( n-2\right) \left( n-6\right) \left( n-8\right) }\left\vert W\left(
p\right) \right\vert ^{2}\lambda ^{4}+o\left( \lambda ^{4}\right) \right) . 
\notag
\end{eqnarray}%
Hence%
\begin{eqnarray}
&&\left\Vert P\varphi _{\lambda }\right\Vert _{L^{\frac{2n}{n+4}}}^{2}
\label{eq3.43} \\
&=&\left( n\left( n+2\right) \left( n-2\right) \left( n-4\right) \right) ^{2}%
\frac{\pi ^{\frac{n+4}{2}}\Gamma \left( \frac{n}{2}\right) ^{\frac{n+4}{n}}}{%
\left( \left( n-1\right) !\right) ^{\frac{n+4}{n}}}\cdot  \notag \\
&&\left( 1-\frac{1}{3}\frac{n^{2}-4n-4}{n\left( n+2\right) \left( n-2\right)
\left( n-6\right) \left( n-8\right) }\left\vert W\left( p\right) \right\vert
^{2}\lambda ^{4}+o\left( \lambda ^{4}\right) \right) .  \notag
\end{eqnarray}%
On the other hand,%
\begin{eqnarray}
&&\int_{M}P\varphi _{\lambda }\cdot \varphi _{\lambda }d\mu  \label{eq3.44}
\\
&=&n\left( n+2\right) \left( n-2\right) \left( n-4\right) \frac{\pi ^{\frac{n%
}{2}}\Gamma \left( \frac{n}{2}\right) }{\left( n-1\right) !}\cdot  \notag \\
&&\left( 1-\frac{n^{2}-4n-4}{6n\left( n+2\right) \left( n-2\right) \left(
n-6\right) \left( n-8\right) }\left\vert W\left( p\right) \right\vert
^{2}\lambda ^{4}+o\left( \lambda ^{4}\right) \right) .  \notag
\end{eqnarray}%
Summing up we get%
\begin{eqnarray}
&&\frac{\int_{M}P\varphi _{\lambda }\cdot \varphi _{\lambda }d\mu }{%
\left\Vert P\varphi _{\lambda }\right\Vert _{L^{\frac{2n}{n+4}}}^{2}}
\label{eq3.45} \\
&=&\Theta _{4}\left( S^{n}\right) \left( 1+\frac{n^{2}-4n-4}{6n\left(
n+2\right) \left( n-2\right) \left( n-6\right) \left( n-8\right) }\left\vert
W\left( p\right) \right\vert ^{2}\lambda ^{4}+o\left( \lambda ^{4}\right)
\right) .  \notag
\end{eqnarray}%
It follows that $\Theta _{4}\left( g\right) >\Theta _{4}\left( S^{n}\right) $%
.

Next we turn to the regularity issue for maximizers of $\Theta _{4}\left(
g\right) $ in (\ref{eq1.16}). Assume $f\in L^{\frac{2n}{n+4}}\left( M\right) 
$, $f\geq 0$ and not identically zero, and it is a maximizer for $\Theta
_{4}\left( g\right) $, then after scaling we have%
\begin{equation}
G_{P}f=\frac{2}{n-4}f^{\frac{n-4}{n+4}}.  \label{eq3.46}
\end{equation}%
Note that this equation is critical in the sense that if we start with $f\in
L^{\frac{2n}{n+4}}$ and use the equation, the usual bootstrap method simply
ends with $f\in L^{\frac{2n}{n+4}}$ again. Approaches in deriving further
regularity for such kind of equations has been well understood (see for
example \cite{DHL,ER,R,V} and so on). Here is a regularity result
particularly tailored for our purpose, we refer the readers to \cite%
{DHL,ER,R,V} for detailed proofs.

\begin{lemma}
\label{lem3.1}Assume $\left( M,g\right) $ is a smooth compact $n$
dimensional Riemannian manifold with $n\geq 5$, $Y\left( g\right) >0$, $%
Q\geq 0$ and not identically zero, $f\in L^{\frac{2n}{n+4}}\left( M\right) $%
, $f\geq 0$ and not identically zero, and it satisfies (\ref{eq3.46}), then $%
f\in C^{\infty }\left( M\right) $, $f>0$.
\end{lemma}

Now we have all the ingredients to prove Theorem \ref{thm1.3}. Theorem \ref%
{thm1.1} clearly follows from Theorem \ref{thm1.3}.

\begin{proof}[Proof of Theorem 1.3]
If $\left( M,g\right) $ is conformal equivalent to the standard sphere, then
everything follows from discussions in Section \ref{sec2.2}. From now on we
assume that $\left( M,g\right) $ is not conformally equivalent to the
standard sphere. By Proposition \ref{prop3.1} we know that $\Theta
_{4}\left( g\right) >\Theta _{4}\left( S^{n}\right) $. \cite[Proposition 1.1]%
{HY4} tells us $\ker P=0$ and $G_{P}>0$. By Proposition \ref{prop2.1} we
know the set%
\begin{equation*}
\mathcal{M}=\left\{ f\in L^{\frac{2n}{n+4}}\left( M\right) :\left\Vert
f\right\Vert _{L^{\frac{2n}{n+4}}\left( M\right) }=1,\int_{M}G_{P}f\cdot
fd\mu =\Theta _{4}\left( g\right) \right\}
\end{equation*}%
is nonempty and compact in $L^{\frac{2n}{n+4}}\left( M\right) $. If $f\in 
\mathcal{M}$, we can assume $f^{+}\neq 0$, then $f^{-}$ must be equal to
zero. Indeed%
\begin{eqnarray*}
&&\Theta _{4}\left( g\right) \\
&=&\int_{M}G_{P}f\cdot fd\mu \\
&=&\int_{M}\left( G_{P}f^{+}\cdot f^{+}-2G_{P}f^{+}\cdot
f^{-}+G_{P}f^{-}\cdot f^{-}\right) d\mu \\
&\leq &\int_{M}G_{P}\left\vert f\right\vert \cdot \left\vert f\right\vert
d\mu \\
&\leq &\Theta _{4}\left( g\right) .
\end{eqnarray*}%
Hence $\int_{M}G_{P}f^{+}\cdot f^{-}d\mu =0$. Using the fact that $G_{P}>0$
and $f^{+}\neq 0$, we see $f^{-}=0$. In another word, $f$ does not change
sign. It follows from Lemma \ref{lem3.1} that $f\in C^{\infty }\left(
M\right) $ and $f>0$. Moreover the compactness of $\mathcal{M}$ under $%
C^{\infty }\left( M\right) $ topology follows from its compactness in $L^{%
\frac{2n}{n+4}}\left( M\right) $ and the proofs of Lemma \ref{lem3.1} in 
\cite{DHL,ER,R,V}.
\end{proof}

\section{Some discussions\label{sec4}}

Here we turn to the variational problem (\ref{eq1.13}).

\begin{proposition}
\label{prop4.1}Let $\left( M,g\right) $ be a smooth compact $n$ dimensional
Riemannian manifold with $n\geq 5$, $Y\left( g\right) >0$, $Q\geq 0$ and not
identically zero, then

\begin{enumerate}
\item $Y_{4}\left( g\right) \leq Y_{4}\left( S^{n}\right) $, here $S^{n}$
has the standard metric. $Y_{4}\left( g\right) =Y_{4}\left( S^{n}\right) $
if and only if $\left( M,g\right) $ is conformally diffeomorphic to the
standard sphere.

\item $Y_{4}\left( g\right) $ is always achieved. Let%
\begin{equation}
\mathcal{M}_{P}=\left\{ u\in H^{2}\left( M\right) :\left\Vert u\right\Vert
_{L^{\frac{2n}{n-4}}\left( M\right) }=1\text{ and }E\left( u\right)
=Y_{4}\left( g\right) \right\} ,  \label{eq4.1}
\end{equation}%
then $\mathcal{M}_{P}$ is not empty. For any $\alpha \in \left( 0,1\right) $%
, $\mathcal{M}_{P}\subset C^{4,\alpha }\left( M\right) $ and when $\left(
M,g\right) $ is not conformally diffeomorphic to the standard sphere, $%
\mathcal{M}_{P}$ is compact under $C^{4,\alpha }$ topology.
\end{enumerate}
\end{proposition}

We start with the following standard fact (see \cite{DHL,He}).

\begin{lemma}
\label{lem4.1}Let%
\begin{equation*}
\mathcal{M}_{P}=\left\{ u\in H^{2}\left( M\right) :\left\Vert u\right\Vert
_{L^{\frac{2n}{n-4}}\left( M\right) }=1\text{ and }E\left( u\right)
=Y_{4}\left( g\right) \right\} .
\end{equation*}%
If $Y_{4}\left( g\right) <Y_{4}\left( S^{n}\right) $, then $\mathcal{M}_{P}$
is nonempty. Moreover for any $\alpha \in \left( 0,1\right) $, $\mathcal{M}%
_{P}\subset C^{4,\alpha }\left( M\right) $ and it is compact in $C^{4,\alpha
}$ topology.
\end{lemma}

\begin{proof}[Proof of Proposition \protect\ref{prop4.1}]
If $\left( M,g\right) $ is conformal equivalent to the standard sphere, then
the conclusion follows from discussions in Section \ref{sec2.2}. Assume $%
\left( M,g\right) $ is not conformal equivalent to the standard sphere, then
it follows from Lemma \ref{lem2.2} and Proposition \ref{prop3.1} that%
\begin{equation*}
Y_{4}\left( g\right) \leq \frac{1}{\Theta _{4}\left( g\right) }<\frac{1}{%
\Theta _{4}\left( S^{n}\right) }=Y_{4}\left( S^{n}\right) .
\end{equation*}%
Here we want to point out that the fact $Y_{4}\left( g\right) <Y_{4}\left(
S^{n}\right) $ can be verified, with the help of positive mass theorem for
Paneitz operator (\cite{HuR,GM,HY4}), by choosing a particular test function
in (\ref{eq1.13}) (see \cite{ER,R,GM}). In fact the corresponding
calculation is easier than what we have in the proof of Proposition \ref%
{prop3.1}, but the statement in Proposition \ref{prop3.1} is stronger. By
Lemma \ref{lem4.1}, we know $\mathcal{M}_{P}$ is nonempty and $\mathcal{M}%
_{P}\subset C^{4,\alpha }\left( M\right) $ and it is compact in $C^{4,\alpha
}\left( M\right) $ for any $\alpha \in \left( 0,1\right) $.
\end{proof}

\begin{proposition}
\label{prop4.2}Let $\left( M,g\right) $ be a smooth compact $n$ dimensional
Riemannian manifold with $n\geq 5$, $Y\left( g\right) >0$, $Y_{4}\left(
g\right) >0$, $Q\geq 0$ and not identically zero. Denote%
\begin{equation*}
\mathcal{M}_{P}=\left\{ u\in H^{2}\left( M\right) :\left\Vert u\right\Vert
_{L^{\frac{2n}{n-4}}\left( M\right) }=1\text{ and }E\left( u\right)
=Y_{4}\left( g\right) \right\}
\end{equation*}%
and%
\begin{equation*}
\mathcal{M}_{\Theta }=\left\{ u\in W^{4,\frac{2n}{n+4}}\left( M\right)
:\left\Vert u\right\Vert _{L^{\frac{2n}{n-4}}\left( M\right) }=1\text{ and }%
\frac{E\left( u\right) }{\left\Vert Pu\right\Vert _{L^{\frac{2n}{n+4}}}^{2}}%
=\Theta _{4}\left( g\right) \right\} .
\end{equation*}%
then

\begin{enumerate}
\item $\mathcal{M}_{P}\subset C^{\infty }\left( M\right) $ and for any $u\in 
\mathcal{M}_{P}$, either $u>0$ or $-u>0$.

\item $Y_{4}\left( g\right) \Theta _{4}\left( g\right) =1$.

\item $\mathcal{M}_{P}=\mathcal{M}_{\Theta }$.
\end{enumerate}
\end{proposition}

\begin{proof}
By Proposition \ref{prop4.1} we know $\mathcal{M}_{P}$ is nonempty and for
any $\alpha \in \left( 0,1\right) $, $\mathcal{M}_{P}\subset C^{4,\alpha
}\left( M\right) $. By \cite[Proposition 1.1]{HY4} we know $G_{P}>0$. Assume 
$u\in \mathcal{M}_{P}$, without losing of generality we can assume $%
u^{+}\neq 0$. Now we will use an observation in \cite{R} to show $u>0$. In
fact $u$ satisfies $\left\Vert u\right\Vert _{L^{\frac{2n}{n-4}}}=1$ and%
\begin{equation*}
Pu=Y_{4}\left( g\right) \left\vert u\right\vert ^{\frac{8}{n-4}}u.
\end{equation*}%
Let $v=G_{P}\left( \left\vert Pu\right\vert \right) $, then $v\in
C^{4,\alpha }\left( M\right) $, $v>0$ and $\left\vert u\right\vert \leq v$.
We have%
\begin{equation*}
Y_{4}\left( g\right) \leq \frac{E\left( v\right) }{\left\Vert v\right\Vert
_{L^{\frac{2n}{n-4}}}^{2}}=Y_{4}\left( g\right) \frac{\int_{M}\left\vert
u\right\vert ^{\frac{n+4}{n-4}}vd\mu }{\left\Vert v\right\Vert _{L^{\frac{2n%
}{n-4}}}^{2}}\leq Y_{4}\left( g\right) \left\Vert v\right\Vert _{L^{\frac{2n%
}{n-4}}}^{-1}\leq Y_{4}\left( g\right) .
\end{equation*}%
Hence all the inequalities become equalities. In particular $\left\Vert
v\right\Vert _{L^{\frac{2n}{n-4}}}=1=\left\Vert u\right\Vert _{L^{\frac{2n}{%
n-4}}}$. Since $v\geq \left\vert u\right\vert $, we see $v=\left\vert
u\right\vert $. This together with $u^{+}\neq 0$ implies $u=v>0$. Standard
bootstrap method shows $u\in C^{\infty }\left( M\right) $. Hence $\mathcal{M}%
_{P}\subset C^{\infty }\left( M\right) $, moreover when $\left( M,g\right) $
is not conformally diffeomorphic to the standard sphere, $\mathcal{M}_{P}$
is compact in $C^{\infty }\left( M\right) $.

For $u\in \mathcal{M}_{P}$, we can assume $u>0$, then $\left\Vert
u\right\Vert _{L^{\frac{2n}{n-4}}}=1$ and%
\begin{equation*}
Pu=Y_{4}\left( g\right) u^{\frac{n+4}{n-4}}.
\end{equation*}%
It follows that from this equation and Lemma \ref{lem2.2} that%
\begin{equation*}
\Theta _{4}\left( g\right) \geq \frac{E\left( u\right) }{\left\Vert
Pu\right\Vert _{L^{\frac{2n}{n+4}}}^{2}}=\frac{1}{Y_{4}\left( g\right) }\geq
\Theta _{4}\left( g\right) .
\end{equation*}%
Hence $Y_{4}\left( g\right) \Theta _{4}\left( g\right) =1$ and $u\in 
\mathcal{M}_{\Theta }$.

On the other hand, if $u\in \mathcal{M}_{\Theta }$, let $f=Pu$, then%
\begin{equation*}
\Theta _{4}\left( g\right) =\frac{\int_{M}Pu\cdot ud\mu }{\left\Vert
Pu\right\Vert _{L^{\frac{2n}{n+4}}}^{2}}=\frac{\int_{M}G_{P}f\cdot fd\mu }{%
\left\Vert f\right\Vert _{L^{\frac{2n}{n+4}}}^{2}}.
\end{equation*}%
Hence it follows from Theorem \ref{thm1.3} that $f\in C^{\infty }\left(
M\right) $ and either $f>0$ or $-f>0$. Without losing of generality we
assume $f>0$, then $u=G_{P}f\in C^{\infty }\left( M\right) $, $u>0$ and%
\begin{equation*}
Pu=\kappa u^{\frac{n+4}{n-4}}
\end{equation*}%
for some positive constant $\kappa $. Using $\left\Vert u\right\Vert _{L^{%
\frac{2n}{n-4}}}=1$ we see%
\begin{equation*}
\Theta _{4}\left( g\right) =\frac{E\left( u\right) }{\left\Vert
Pu\right\Vert _{L^{\frac{2n}{n+4}}}^{2}}=\frac{1}{\kappa },
\end{equation*}%
and hence $\kappa =Y_{4}\left( g\right) $. It follows that $E\left( u\right)
=Y_{4}\left( g\right) $ and hence $u\in \mathcal{M}_{P}$. Summing up we see $%
\mathcal{M}_{P}=\mathcal{M}_{\Theta }$.
\end{proof}

Now we are ready to prove Theorem \ref{thm1.2}.

\begin{proof}[Proof of Theorem \protect\ref{thm1.2}]
It is clear Theorem \ref{thm1.2} follows from Proposition \ref{prop4.1} and %
\ref{prop4.2}. The compactness of $\mathcal{M}_{P}$ in $C^{\infty }$
topology was shown in the proof of Proposition \ref{prop4.2}.
\end{proof}

At last we will show the approach to the $Q$ curvature equation in Theorem %
\ref{thm1.3} gives another way to find constant scalar curvature metric in a
conformal class with positive Yamabe invariant. Here we always assume $%
\left( M,g\right) $ is a smooth compact $n$ dimensional Riemannian manifold
with $n\geq 3$ and $Y\left( g\right) >0$. To find a conformal metric with
scalar curvature $1$ is the same as solving%
\begin{equation}
L\rho =\rho ^{\frac{n+2}{n-2}},\quad \rho \in C^{\infty }\left( M\right)
,\rho >0.  \label{eq4.2}
\end{equation}%
Here $L$ is the conformal Laplacian operator. For any $u\in C^{\infty
}\left( M\right) $ we write%
\begin{eqnarray}
E_{2}\left( u\right) &=&\int_{M}Lu\cdot ud\mu  \label{eq4.3} \\
&=&\int_{M}\left( \frac{4\left( n-1\right) }{n-2}\left\vert \nabla
u\right\vert ^{2}+Ru^{2}\right) d\mu .  \notag
\end{eqnarray}%
Clearly $E_{2}\left( u\right) $ extends continuously to $u\in H^{1}\left(
M\right) $. To solve (\ref{eq4.2}), people consider the variational problem
(see \cite{LP})%
\begin{equation}
Y\left( g\right) =\inf_{u\in H^{1}\left( M\right) \backslash \left\{
0\right\} }\frac{E_{2}\left( u\right) }{\left\Vert u\right\Vert _{L^{\frac{2n%
}{n-2}}}^{2}}.  \label{eq4.4}
\end{equation}%
Denote%
\begin{equation}
\mathcal{M}_{L}=\left\{ u\in H^{1}\left( M\right) :\left\Vert u\right\Vert
_{L^{\frac{2n}{n-2}}}=1\text{ and }E_{2}\left( u\right) =Y\left( g\right)
\right\} ,  \label{eq4.5}
\end{equation}%
then it is well known that $\mathcal{M}_{L}$ is always nonempty, $\mathcal{M}%
_{L}\subset C^{\infty }\left( M\right) $ and for any $u\in \mathcal{M}_{L}$,
either $u>0$ or $-u>0$. If $u>0$, then after scaling $u$ solves (\ref{eq4.2}%
). Moreover when $\left( M,g\right) $ is not conformally diffeomorphic to
the standard sphere, we have $Y\left( g\right) <Y\left( S^{n}\right) $ and $%
\mathcal{M}_{L}$ is compact in $C^{\infty }$ topology (see \cite{LP,S}).

Now we turn to another approach to solve (\ref{eq4.2}). Since $Y\left(
g\right) >0$, we know the Green's function of $L$ exists and it is always
positive. We can define an operator%
\begin{equation}
\left( G_{L}f\right) \left( p\right) =\int_{M}G_{L}\left( p,q\right) f\left(
q\right) d\mu \left( q\right) .  \label{eq4.6}
\end{equation}%
Let $f=\rho ^{\frac{n+2}{n-2}}$, then (\ref{eq4.2}) becomes%
\begin{equation}
G_{L}f=f^{\frac{n+2}{n-2}},\quad f\in C^{\infty }\left( M\right) ,f>0.
\label{eq4.7}
\end{equation}%
Let%
\begin{equation}
\Theta _{2}\left( g\right) =\sup_{f\in L^{\frac{2n}{n+2}}\left( M\right)
\backslash \left\{ 0\right\} }\frac{\int_{M}G_{L}f\cdot fd\mu }{\left\Vert
f\right\Vert _{L^{\frac{2n}{n+2}}}^{2}}=\sup_{u\in W^{2,\frac{2n}{n+2}%
}\left( M\right) \backslash \left\{ 0\right\} }\frac{\int_{M}Lu\cdot ud\mu }{%
\left\Vert Lu\right\Vert _{L^{\frac{2n}{n+2}}}^{2}}.  \label{eq4.8}
\end{equation}%
Note that this functional is considered in \cite{DoZ}. Same argument as in
the proof of Lemma \ref{lem2.1} shows%
\begin{equation}
\Theta _{2}\left( g\right) =\sup_{\widetilde{g}\in \left[ g\right] }\frac{%
\int_{M}\widetilde{R}d\widetilde{\mu }}{\left\Vert \widetilde{R}\right\Vert
_{L^{\frac{2n}{n+2}}\left( M,d\widetilde{\mu }\right) }^{2}}.  \label{eq4.9}
\end{equation}%
With the solution to Yamabe problem (\cite{LP,S}) we can deduce

\begin{lemma}
\label{lem4.2}Let $\left( M,g\right) $ be a smooth compact $n$ dimensional
Riemannian manifold with $n\geq 3$, $Y\left( g\right) >0$. Denote%
\begin{equation*}
\mathcal{M}_{\Theta _{2}}=\left\{ u\in W^{2,\frac{2n}{n+2}}\left( M\right)
:\left\Vert u\right\Vert _{L^{\frac{2n}{n-2}}\left( M\right) }=1\text{ and }%
\frac{E_{2}\left( u\right) }{\left\Vert Lu\right\Vert _{L^{\frac{2n}{n+2}%
}}^{2}}=\Theta _{2}\left( g\right) \right\} .
\end{equation*}%
Then

\begin{enumerate}
\item $Y\left( g\right) \Theta _{2}\left( g\right) =1$.

\item $\mathcal{M}_{L}=\mathcal{M}_{\Theta _{2}}$.
\end{enumerate}
\end{lemma}

Since the proof is essentially the same as the one for Proposition \ref%
{prop4.2}, we omit it here. Roughly speaking Lemma \ref{lem4.2} tells us the
maximization problem for $\Theta _{2}\left( g\right) $ will not produce new
constant scalar curvature metrics other than those by minimizing problem for 
$Y\left( g\right) $. However, \textit{without using the solution to Yamabe
problem}, we can use the same argument as for Theorem \ref{thm1.3} to show $%
\Theta _{2}\left( g\right) \geq \Theta _{2}\left( S^{n}\right) $, with
equality holds if and only if $\left( M,g\right) $ is conformally
diffeomorphic to the standard sphere (here one needs to use the positive
mass theorem); $\mathcal{M}_{\Theta _{2}}$ is always nonempty, $\mathcal{M}%
_{\Theta _{2}}\subset C^{\infty }\left( M\right) $ and any $u\in \mathcal{M}%
_{\Theta _{2}}\mathcal{\ }$must be either positive or negative; $\mathcal{M}%
_{\Theta _{2}}$ is compact in $C^{\infty }\left( M\right) $ when $\left(
M,g\right) $ is not conformally diffeomorphic to the standard sphere. In
particular, this gives another way to solve (\ref{eq4.2}). The details are
left to interested readers.

\end{document}